\documentclass[12pt]{amsart}
\usepackage{amsmath}
\usepackage{amssymb}
\usepackage{amsfonts}
\usepackage{amsthm}

\newtheorem{theo}{Theorem}[section]
\newtheorem*{theo*}{Theorem}

\newtheorem{lemm}[theo]{Lemma}

\newcommand{\C}{\mathbb{C}}
\newcommand{\lchi}{|L(1,\chi)|}

\title[]{On two mean square averages of Dirichlet $L$-function}
\begin{document}
 
 \keywords{$L$-functions, trigonometric sums, Jordan totient function, Euler totient function, mean square averages, Gauss sum, Ramanujan sum, Bernoulli numbers}
 \subjclass[2010]{11M06, 11L05, 11L03}
 
 \author[N E Thomas]{Neha Elizabeth Thomas}
 \address{Department of Mathematics, University College, Thiruvananthapuram, Kerala - 695034, India}
 \email{nehathomas2009@gmail.com}
 \author[A Chandran]{Arya Chandran}
 \address{Department of Mathematics, University College, Thiruvananthapuram, Kerala - 695034, India}
 \email{aryavinayachandran@gmail.com}
\author[K V Namboothiri]{K Vishnu Namboothiri}
\address{Department of Mathematics, Government College, Ambalapuzha, (Affiliated to the University of Kerala, Thiruvananthapuram) Kerala - 688561, INDIA\\Department of Collegiate Education, Government of Kerala, India}
\email{kvnamboothiri@gmail.com}

 \begin{abstract}
Finding the mean square averages of the Dirichlet $L$-functions over Dirichlet characters $\chi$ of same parity is an active problem in number theory. Here we explicitly evaluate such averages of $L(3,\chi)$ and $L(4,\chi)$ using certain trigonometric sums and Bernoulli polynomials and express them in terms of the Euler totient function $\phi$ and the Jordan totient function $J_s$.
 \end{abstract}

 \maketitle
\section{Introduction}
Let $k$ be a natural number $\geq 3$. A Dirichlet character $\chi$ is defined to be odd if $\chi(-1)=-1$ and even if $\chi(-1)=1$. The Dirichlet $L$-function $L(s,\chi)$ is defined by the infinite series $\sum\limits_{n=1}^{\infty}\frac{\chi(n)}{n^s}$ where $s\in\C$ with $Re\,(s)>1$. It is an important function in number theory especially due to its connection with the Rieman zeta function $\zeta(s)$. For rational integer $r$, the problem of computing exact values of 
\begin{align}\label{eqn:gen_sum}
 \sum\limits_{\substack{\chi \text{ mod }k\\\chi(-1)=(-1)^r}}|L(r,\chi)|^2
\end{align}
and thus finding the mean square averages of this sum has been attempted in various cases by many.  

 In 1982, Walum \cite{walum1982exact} gave an exact formula for the sum (\ref{eqn:gen_sum}) with $r=1$. Louboutin (\cite{louboutin1999mean}) computed  the sum of $\lchi^2$  over all odd primitive Dirichlet characters modulo $k$. See \cite[Chapter 6]{tom1976introduction} for the definition of primitivity of Dirichlet characters. In \cite{louboutin1993quelques}, Louboutin gave an exact formula for the sum of $\lchi^2$ over all odd Dirichlet characters in terms of the prime divisors of $k$ and the Euler totient function $\phi$.  He mainly used the orthogonality properties of characters and  some trigonometric identites in his computations. Using the same techniques, in \cite{louboutin2001mean}  he derived exact  formulae for the  general versions of these sums in two cases : $\chi$ even and $\chi$ odd. 

E. Alkan in \cite{alkan2011mean} derived exact formulae for the sums $\sum\limits_{\chi \, \text{odd}}|L(1,\chi)|^2$ and $\sum\limits_{\chi \, \text{even}}|L(2,\chi)|^2$ using weighted averages of Gauss and Ramanujan sums. His formulae involved Jordan totient function and Euler totient function (See the next section for definitions). Alkan employed certain exact evaluations of trigonometric sums appearing in \cite{alkan2011values} to enable these computations. The computations he performed in this paper were so extensive and beautiful so that some of the identities he derived during these computations gave rise a sequence of papers starting with those by L. Toth \cite{toth2014averages} and K V Nambooothiri \cite{namboothiri2017certain}.

After Alkan, many other authors also attempted the problem of finding mean square values of $L(r,\chi)$ for various values of $r$. A general formula was provided by T. Okamoto and T. Onozuka in \cite{okamoto2015mean} and \cite{okamoto2017mean} using a technique suggested by S. Louboutin in \cite{louboutin2001mean}. Note that this technique was different from the one used by Alkan in \cite{alkan2011mean}. The mean square values and certain related problems were dealt in some other papers as well, see example,   \cite{wu2012mean},  \cite{zhang2015hybrid}, and the recent papers \cite{lin2019mean}, \cite{zhang2020certain}. In \cite{lin2019mean}, X. Lin provided a general inductive formula for computing the sum (\ref{eqn:gen_sum}).

Some other closely related problems were attempted by a few authors. In \cite{zhang2003problem}, W. Zhang derived an asymptotic (not exact) formula for $\lchi^4$ over odd Dirichlet characters. He used Abel's identity and Cauchy's inequality in addition to the orthogonality properties of characters to arrive at his estimate. Alkan in \cite{alkan2013averages} derived the sum $\sum\limits_{\chi \, \text{odd}}(L(1,\chi))^r$ where $r\geq 1$. In \cite{louboutin2015twisted}, S. Louboutin computed $\sum\limits_{\chi \, \text{odd}}\chi(c)|L(1,\chi)|^2$ where c, a positive integer with some extra conditions. Computation of such sums, known as twisted sums is another problem of active interest.

Note that Alkan's techniques in \cite{alkan2011mean} were completely different from the other derivations used in computing the square sums. It was also observed by Alkan that his techniques could be used to determine the sum (\ref{eqn:gen_sum}) for any larger $r(\geq 3)$, but at the cost of increasingly complex computations. We here undertake these complex comuptations and use the same techniques used by Alkan in \cite{alkan2011mean} to derive exact formulae for $\sum\limits_{\chi(-1)=-1}|L(3,\chi)|^2$ and $\sum\limits_{\chi(-1)=1}|L(4,\chi)|^2$.
We also derive some trigonometric identities during these computations that could be of independent interest.

 \section{Notations and some elementary results}
 In this section we introduce the basic definitions and some identities that we use throughout this paper. The definitions of terms we do not define, but appear in this paper can be found in \cite{tom1976introduction} or \cite{montgomery2006multiplicative}. For a positive integer $k$, if $\chi$ is a  Dirichlet character  modulo $k$, then $\chi(-1)=(-1)^n$ for some natural number $n$. The parity of $\chi$ is the parity of this $n$. Hence $\chi$ is odd if $n$ is odd and  even otherwise.
 
 The Gauss sum $G(z,\chi)$ for any complex number $z$ is defined as 
 
 \begin{align}
 G(z,\chi):= \sum\limits_{m=1}^{k}\chi(m)e^{\frac{2\pi imz}{k}}.
 \end{align} When $\chi=\chi_0$ is  the principal character, Gauss sum becomes the Ramanujan sum $R_k(z)$:
 \begin{align}
  R_k(z):= \sum\limits_{\substack{m=1\\(m, k)=1}}^{k}e^{\frac{2\pi imz}{k}}.
 \end{align}

 The Jordan totient function $J_k(n)$ is defined by  \begin{align}
J_k(n) := n^k\prod_{\substack{p|n\\p\text{ prime}}}\left(1-\frac{1}{p^k}\right).
\end{align} 

 A similar type of function which we use in this paper is $\phi_k(n)$,   defined as the sum of the $k$th powers of numbers $\leq n$ and relatively prime to $n$. 
 
By $B_q$ we mean the $q$\textsuperscript{th} Bernoulli number. See \cite[Chapter 12]{tom1976introduction} for definition and other properties of Bernoulli numbers and Bernoulli polynomials. The values of Bernoulli numbers we use in this paper are  $B_0=1$, $B_1=-\frac{1}{2}$, $B_2=\frac{1}{6}$,   $B_3=0$, $B_4=-\frac{1}{30}$.

For a positive integer $m$, by $S(m, \chi)$ we mean the sum
\begin{align}                                 
S(m, \chi):=\sum\limits_{j=1}^k(\frac{j}{k})^mG(j, \chi).
\end{align}
 Some other important identities that we use in our computation are listed below:

\begin{align}
& \sum\limits_{\substack{1\leq m\leq k \\ (m, k)=1}} \frac{1}{\sin^2(\frac{\pi m}{k})}=\frac{J_2(k)}{3} \label{eqn:1bysin2}\text{ \cite[Identity 5.16]{alkan2011values}}.\\
  &\sum\limits_{j=1}^{k-1}j^re^{\frac{2\pi imj}{k}}=\sum\limits_{j=1}^{k-1}\binom{r}{j} k^j\lim\limits_{w\rightarrow \frac{2\pi im}{k}}\frac{d^{r-j}}{dw^{r-j}} \left(\frac{1}{e^w-1}\right) \label{eqn:jrediff}\text{ \cite[Section 2]{alkan2011values}}.
 \end{align}
 The following identities can be easily computed using (\ref{eqn:jrediff}): 
 \begin{align}
  &\sum\limits_{s=1}^{k-1}se^{\frac{2\pi ims}{k}}=\frac{k}{e^{\frac{2\pi im}{k}}-1}.\\
 &\sum\limits_{s=1}^{k-1}se^{\frac{-2\pi ims}{k}}=-k-\frac{k}{e^{\frac{2\pi im}{k}}-1}.\\
 &\sum\limits_{s=1}^{k-1}s^2e^{\frac{2\pi ims}{k}}=\frac{k^2}{e^{\frac{2\pi im}{k}}-1}-\frac{2ke^{\frac{2\pi im}{k}}}{(e^{\frac{2\pi im}{k}}-1)^2}.\\
& \sum\limits_{s=1}^{k-1}s^2e^{\frac{-2\pi ims}{k}}=-k^2-\frac{k^2}{e^{\frac{2\pi im}{k}}-1}-\frac{2ke^{\frac{2\pi im}{k}}}{(e^{\frac{2\pi im}{k}}-1)^2}.
 \end{align}
  \begin{align}\label{1}
 \frac{(-1)^{v+1}kr!}{i^r2^{r-1}\pi^r}L(r,\chi)=\sum\limits_{q=0}^{2[\frac{r}{2}]}\binom{r}{q}
 B_q S(r-q,\chi)\text{ \cite[Theorem 1]{alkan2011values}  } 
 \end{align}
 where $ \chi$  and  $r\geq1$ have same parity. 
 
Now we may derive a useful identity quickly. Using identity (5.12) in \cite{alkan2011values} we have
  \begin{align*}
   \frac{k^4}{\pi^4}L(4, \chi_0)
&=\frac{2}{3}\sum\limits_{\substack{1\leq m\leq k \\ (m, k)=1}}
\frac{1}{\sin^2(\frac{\pi m}{k})}+\frac{1}{2}\sum\limits_{\substack{1\leq m\leq k \\ (m, k)=1}}
\frac{\cos(\frac{2\pi m}{k})}{\sin^4(\frac{\pi m}{k})}\\
&=\frac{2}{3}\sum\limits_{\substack{1\leq m\leq k \\ (m, k)=1}}
\frac{1}{\sin^2(\frac{\pi m}{k})}+\frac{1}{2}\sum\limits_{\substack{1\leq m\leq k \\ (m, k)=1}}
\frac{(1-2\sin^2(\frac{\pi m}{k}))}{\sin^4(\frac{\pi m}{k})}\\
&=\frac{1}{2}\sum\limits_{\substack{1\leq m\leq k \\ (m, k)=1}}
\frac{1}{\sin^4(\frac{\pi m}{k})}-\frac{1}{3}\sum\limits_{\substack{1\leq m\leq k \\ (m, k)=1}}
\frac{1}{\sin^2(\frac{\pi m}{k})}.
 \end{align*}
 Using $ \frac{k^4}{\pi^4}L(4, \chi_0)=\frac{k^4}{\pi^4}\zeta(4)\prod\limits_{\substack{p|n\\p\text{ prime}}}\left(1-\frac{1}{p^4}\right)=\frac{J_4(k)}{90}$ \cite[Chapter 11]{tom1976introduction} and identity (\ref{eqn:1bysin2}) above we get 
 \begin{align}\label{eqn:1bysin4}
  \sum\limits_{\substack{1\leq m\leq k \\ (m, k)=1}}\frac{1}{\sin^4(\frac{\pi m}{k})}=\frac{J_4(k)}{45}+\frac{2}{9}J_2(k).
  \end{align}
  
We now state four identities. 
\begin{lemm}
\begin{align}
&\sum\limits_{j=1}^{k-1}j^3e^{\frac{2\pi imj}{k}}=\frac{k^3}{e^{\frac{2\pi im}{k}}-1}-\frac{3k^2e^{\frac{2\pi im}{k}}}{(e^{\frac{2\pi im}{k}}-1)^2}+\frac{3ke^{\frac{2\pi im}{k}}+3ke^{\frac{4\pi im}{k}}}{(e^{\frac{2\pi im}{k}}-1)^3}.\label{j^3}\\
&\sum\limits_{j=1}^{k-1}j^3e^{\frac{-2\pi imj}{k}}=-\frac{k^3e^{\frac{2\pi im}{k}}}{e^{\frac{2\pi im}{k}}-1}-\frac{3k^2e^{\frac{2\pi im}{k}}}{(e^{\frac{2\pi im}{k}}-1)^2}-\frac{(3ke^{\frac{2\pi im}{k}}+3ke^{\frac{4\pi im}{k}})}{(e^{\frac{2\pi im}{k}}-1)^3}.\label{j^-3}\\
&\sum\limits_{j=1}^{k-1}j^4e^{\frac{2\pi imj}{k}}\label{j^4}=\frac{k^4}{e^{\frac{2\pi im}{k}}-1}-\frac{4k^3e^{\frac{2\pi im}{k}}}{(e^{\frac{2\pi im}{k}}-1)^2}+\frac{6k^2e^{\frac{2\pi im}{k}}+6k^2e^{\frac{4\pi im}{k}}}{(e^{\frac{2\pi im}{k}}-1)^3}\\&-\frac{(4ke^{\frac{2\pi im}{k}}+16ke^{\frac{4\pi im}{k}}+4ke^{\frac{6\pi im}{k}})}{(e^{\frac{2\pi im}{k}}-1)^4}\nonumber. \\
&\sum\limits_{j=1}^{k-1}j^4e^{\frac{-2\pi imj}{k}}\label{j^-4}=-\frac{k^4e^{\frac{2\pi im}{k}}}{e^{\frac{2\pi im}{k}}-1}-\frac{4k^3e^{\frac{2\pi im}{k}}}{(e^{\frac{2\pi im}{k}}-1)^2}-\frac{(6k^2e^{\frac{2\pi im}{k}}+6k^2e^{\frac{4\pi im}{k}})}{(e^{\frac{2\pi im}{k}}-1)^3}\\&-\frac{(4ke^{\frac{2\pi im}{k}}+16ke^{\frac{4\pi im}{k}}+4ke^{\frac{6\pi im}{k}})}{(e^{\frac{2\pi im}{k}}-1)^4} \nonumber.
\end{align}
\end{lemm}
\begin{proof}
The identity (\ref{j^3}) can be derived by putting $r=3$ in  (\ref{eqn:jrediff}) and by applying the limit.

Now we derive the second identity as follows:
\begin{align*}
\sum\limits_{j=1}^{k-1}j^3e^{\frac{-2\pi imj}{k}}
&=\sum\limits_{j=1}^{k-1}j^3e^{\frac{2\pi i(k-m)j}{k}}\\
&=\frac{k^3}{e^{\frac{2\pi i(k-m)}{k}}-1}-\frac{3	k^2	e^{\frac{2\pi i(k-m)}{k}}}{(e^{\frac{2\pi i(k-m)}{k}}-1)^2}+\frac{3ke^{\frac{2	\pi i(k-m)}{k}}+3ke^{\frac{4\pi i(k-m)}{k}}}{(e^{\frac{2\pi i(k-m)}{k}}-1)^3}\\
&=\frac{k^3}{e^{\frac{-2\pi im}{k}}-1}-\frac{3	k^2	e^{\frac{-2\pi im}{k}}}{(e^{\frac{-2\pi im}{k}}-1)^2}+\frac{3ke^{\frac{-2	\pi im}{k}}+3ke^{\frac{-4\pi im}{k}}}{(e^{\frac{-2\pi im}{k}}-1)^3}\\
&=-\frac{k^3}{e^{\frac{2\pi im}{k}}-1}-\frac{3	k^2	e^{\frac{2\pi im}{k}}}{(e^{\frac{2\pi im}{k}}-1)^2}-\frac{(3ke^{\frac{2	\pi im}{k}}+3ke^{\frac{4\pi im}{k}})}{(e^{\frac{2\pi im}{k}}-1)^3}.
\end{align*}
The last line above is obtained from the fact that $$e^{\frac{-2\pi im}{k}}-1=\frac{1-e^{\frac{2\pi im}{k}}}{e^{\frac{2\pi im}{k}}}=-\frac{(e^{\frac{2\pi im}{k}}-1)}{e^{\frac{2\pi im}{k}}}.$$
\end{proof}
 \section{Main Results and proofs}
We state below the main results we prove in this paper. 

\begin{theo}\label{L(3,x)}
The formula
\begin{align*}
\sum\limits_{\substack{\chi(\text{mod }  k) \\ \chi \text{ odd}}}|L(3, \chi)|^2 = \frac{\pi^6}{90k^6}\phi(k)\left(\frac{J_6(k)}{21}-J_2(k)
\right)
\end{align*} 
holds for all $k\geq3$.

\end{theo}

\begin{theo}\label{L(4,x)}
The formula
\begin{align*}
\sum\limits_{\substack{\chi(\text{mod }  k) \\ \chi \text{ even}}}|L(4, \chi)|^2 = \frac{\pi^8}{27k^8}\phi(k)\left(\frac{J_8(k)}{700}+\frac{J_4(k)}{150}+\frac{2}{21}J_2(k)
\right)
\end{align*} 
holds for all $k\geq3$.
 
\end{theo}

From the above two results, it can be seen that the average values of $|L(3, \chi)|^2$ over all odd characters modulo $k$ is
$\frac{2\pi^6}{90k^6}\left(\frac{J_6(k)}{21}-J_2(k)\right)$ and $|L(4, \chi)|^2$ over all even characters modulo $k$ is
$\frac{2\pi^8}{27k^8}\left(\frac{J_8(k)}{700}+\frac{J_4(k)}{150}+\frac{2}{21}J_2(k)\right)$.

Now we proceed to prove the above two results. The outline of the proof is the following:
\begin{enumerate}
 \item Use identity (\ref{1}) connecting $L(r,\chi)$, Bernoulli numbers, and Gauss sums with $r=3, 4$ and expand it.
 \item Take the sum of $|L(r,\chi)|^2$ over $\chi$ having same parity as that of $r$, further expand it using the complex number property $|z|^2=z\overline{z}$.
 \item The expansion consists of terms of the form $\sum\limits_{j=1}^{k-1} \sum\limits_{s=1}^{k-1}j^p s^q\sum\limits_{\substack{1\leq m\leq k \\ (m, k)=1}}e^{\frac{2\pi im(j\pm s)}{k}}$ for certain values of $p,q$. Simplify these terms further. 
 \item In the simplification of terms obtained in the last step, the exponential power expands to trigonometric identities, necessiating the computation of sums of the trigonometric functions of the form $\sum\limits_{\substack{1\leq m\leq k \\ (m, k)=1}}\frac{1}{\sin^p(\frac{\pi m}{k})}$ where $p$ takes values 2, 4,\ldots. These sums results in certain combinations of the Euler totient function and Jordan totient function which is what we intend to prove.
\end{enumerate}

\begin{proof}[Proof of theorem \ref{L(3,x)}]

 Put $r=3$ in equation (\ref{1}) and use the values of the Bernoulli numbers to get
\begin{align}\label{2}
L(3,\chi)=\frac{-2i\pi^3}{3k}
\left[
S(3, \chi)-\frac{3}{2}S(2,\chi)+\frac{1}{2}S(1, \chi)
\right].
\end{align}
 If $\chi\neq\chi_0$ and $r\geq1$ are of opposite parity, then
 \begin{align}\label{3}
 \sum\limits_{q=0}^{2[\frac{r}{2}]}\binom{r}{q}
  B_q S(r-q.\chi)=0 \text{  \cite[Theorem 1]{alkan2011values} }.
 \end{align}
When $r=2$,  we get 
\begin{align}
S(2, \chi)-S(1, \chi)+\frac{1}{6}S(0, \chi) =0.\label{4}
\end{align} 
 Since $S(0, \chi)=\sum\limits_{j=1}^kG(j,\chi)=0$,  
  \begin{align}\label{5}
 S(2, \chi)=S(1, \chi).
 \end{align}
 From equations (\ref{2}) and (\ref{5}) we have
 \begin{align*}
L(3,\chi)=\frac{-2i\pi^3}{3k}
\left[
S(3, \chi)-S(1, \chi)
\right].
\end{align*}
 Since $G(k, \chi)= \sum\limits_{m=1}^{k}\chi(m)=0$ for any non principal character  \cite[Theroem 6.10]{tom1976introduction}, we have
 \begin{align}\label{6}
  L(3, \chi)= \frac{-2i\pi^3}{3k^2}
\left[
\frac{1}{k^2}\sum\limits_{j=1}^{k-1} j^3G(j, \chi)-\sum\limits_{j=1}^{k-1} jG(j, \chi)
\right].
 \end{align}
 Now we are ready to compute the mean square sum.
 Using the above identity we get
 \begin{align*}
 &\sum\limits_{\substack{\chi(\text{mod }  k) \\ \chi \text{ odd}}}|L(3, \chi)|^2 \\
 =& \frac{4\pi^6}{9k^4}\sum\limits_{\substack{\chi(\text{mod }  k) \\ \chi \text{ odd}}}
 \left|
 \frac{1}{k^2}\sum\limits_{j=1}^{k-1} j^3G(j, \chi)-\sum\limits_{j=1}^{k-1} jG(j, \chi)
 \right|^2\\
 =&\frac{4\pi^6}{9k^4}\sum\limits_{\substack{\chi(\text{mod }  k) \\ \chi \text{ odd}}}
 \left(
 \frac{1}{k^2}\sum\limits_{j=1}^{k-1} j^3G(j, \chi)-\sum\limits_{j=1}^{k-1} jG(j, \chi)
 \right) \times\\
 &\left(
 \frac{1}{k^2}\sum\limits_{s=1}^{k-1} s^3\overline{G(s, \chi)}-\sum\limits_{s=1}^{k-1} s\overline{G(s, \chi)}
 \right)
 \\
  =&\frac{4\pi^6}{9k^8}\sum\limits_{j=1}^{k-1} \sum\limits_{s=1}^{k-1}j^3s^3\sum\limits_{m=1}^{k-1}\sum\limits_{n=1}^{k-1}e^{\frac{2\pi i(mj-ns)}{k}}\sum\limits_{\substack{\chi(\text{mod }  k) \\ \chi \text{ odd}}}\chi(m)\overline{\chi(n)}
 \\
-& \frac{4\pi^6}{9k^6}\sum\limits_{j=1}^{k-1} \sum\limits_{s=1}^{k-1}j^3s\sum\limits_{m=1}^{k-1}\sum\limits_{n=1}^{k-1}e^{\frac{2\pi i(mj-ns)}{k}}\sum\limits_{\substack{\chi(\text{mod }  k) \\ \chi \text{ odd}}}\chi(m)\overline{\chi(n)}\\
-&\frac{4\pi^6}{9k^6}\sum\limits_{j=1}^{k-1} \sum\limits_{s=1}^{k-1}js^3\sum\limits_{m=1}^{k-1}\sum\limits_{n=1}^{k-1}e^{\frac{2\pi i(mj-ns)}{k}}\sum\limits_{\substack{\chi(\text{mod }  k) \\ \chi \text{ odd}}}\chi(m)\overline{\chi(n)}\\
+&\frac{4\pi^6}{9k^4}\sum\limits_{j=1}^{k-1} \sum\limits_{s=1}^{k-1}js\sum\limits_{m=1}^{k-1}\sum\limits_{n=1}^{k-1}e^{\frac{2\pi i(mj-ns)}{k}}\sum\limits_{\substack{\chi(\text{mod }  k) \\ \chi \text{ odd}}}\chi(m)\overline{\chi(n)}.
 \end{align*}
There are exactly $\frac{\phi(k)}{2}$ odd characters modulo $k\geq3$. Therefore $\sum\limits_{\substack{\chi(\text{mod }  k) \\ \chi \text{ odd}}}\chi(u)=\frac{\phi(k)}{2}$ or $-\frac{\phi(k)}{2}$ depending on if $u=1$ or $u=-1$ respectively. If $u$ is neither of these modulo $k$, then  $\sum\limits_{\substack{\chi(\text{mod }  k) \\ \chi \text{ odd}}}\chi(u)=0$. Thus 

 $\sum\limits_{\substack{\chi(\text{mod }  k) \\ \chi \text{ odd}}}\chi(m)\overline{\chi(n)}=\sum\limits_{\substack{\chi(\text{mod }  k) \\ \chi \text{ odd}}}\chi(mn^{-1})=\frac{\phi(k)}{2} \text{ or } -\frac{\phi(k)}{2}$
 
when $m=n$ or $m=k-n$ repectively and  
 $\sum\limits_{\substack{\chi(\text{mod }  k) \\ \chi \text{ odd}}}\chi(m)\overline{\chi(n)}=0 $
 
 otherwise.
 
 Hence we get
 \begin{align}\label{|l(3,x)|}
  &\sum\limits_{\substack{\chi(\text{mod }  k) \\ \chi \text{ odd}}}|L(3, \chi)|^2 \\ =&\frac{4\pi^6}{9k^8}\frac{\phi(k)}{2}
  \left(
  \sum\limits_{j=1}^{k-1} \sum\limits_{s=1}^{k-1}j^3s^3\sum\limits_{\substack{1\leq m\leq k \\ (m, k)=1}}e^{\frac{2\pi im(j-s)}{k}}-\sum\limits_{j=1}^{k-1} \sum\limits_{s=1}^{k-1}j^3s^3\sum\limits_{\substack{1\leq m\leq k \\ (m, k)=1}}e^{\frac{2\pi im(j+s)}{k}}
 \right) \nonumber \\ 
  -&\frac{4\pi^6}{9k^6}\frac{\phi(k)}{2}
  \left(
  \sum\limits_{j=1}^{k-1} \sum\limits_{s=1}^{k-1}j^3s\sum\limits_{\substack{1\leq m\leq k \\ (m, k)=1}}e^{\frac{2\pi im(j-s)}{k}}-\sum\limits_{j=1}^{k-1} \sum\limits_{s=1}^{k-1}j^3s\sum\limits_{\substack{1\leq m\leq k \\ (m, k)=1}}e^{\frac{2\pi im(j+s)}{k}}
  \right)\nonumber \\
  -&\frac{4\pi^6}{9k^6}\frac{\phi(k)}{2}
  \left(
  \sum\limits_{j=1}^{k-1} \sum\limits_{s=1}^{k-1}js^3\sum\limits_{\substack{1\leq m\leq k \\ (m, k)=1}}e^{\frac{2\pi im(j-s)}{k}}-\sum\limits_{j=1}^{k-1} \sum\limits_{s=1}^{k-1}js^3\sum\limits_{\substack{1\leq m\leq k \\ (m, k)=1}}e^{\frac{2\pi im(j+s)}{k}}
  \right)\nonumber \\
  +&\frac{4\pi^6}{9k^4}\frac{\phi(k)}{2}
  \left(
  \sum\limits_{j=1}^{k-1} \sum\limits_{s=1}^{k-1}js\sum\limits_{\substack{1\leq m\leq k \\ (m, k)=1}}e^{\frac{2\pi im(j-s)}{k}}-\sum\limits_{j=1}^{k-1} \sum\limits_{s=1}^{k-1}js\sum\limits_{\substack{1\leq m\leq k \\ (m, k)=1}}e^{\frac{2\pi im(j+s)}{k}}
  \right).\nonumber
 \end{align}
  
  Take the first term in the above expression. If we write $\alpha=e^{\frac{2\pi im}{k}}$ for notational convenience, we have
  \begin{align*}
 &\sum\limits_{j=1}^{k-1} \sum\limits_{s=1}^{k-1}j^3s^3\sum\limits_{\substack{1\leq m\leq k \\ (m, k)=1}}e^{\frac{2\pi im(j-s)}{k}}\\
   =& \sum\limits_{\substack{1\leq m\leq k \\ (m, k)=1}}
 \left(\sum\limits_{j=1}^{k-1} j^3\alpha^j
\right)
\left( 
 \sum\limits_{s=1}^{k-1}s^3\alpha^{-s}
 \right)\\
 =& \sum\limits_{\substack{1\leq m\leq k \\ (m, k)=1}}
 \left(
 \frac{k^3}{\alpha-1}-\frac{3k^2\alpha}{(\alpha-1)^2}+\frac{3k\alpha^2+3k\alpha}{(\alpha-1)^3}
 \right) \times \\
 &\left(
 -\frac{k^3\alpha}{\alpha-1}-\frac{3k^2\alpha}{(\alpha-1)^2}-\frac{3k\alpha^2+3k\alpha}{(\alpha-1)^3}
 \right)\\ &(\text{ using identities }(\ref{j^3}) \text{ and }(\ref{j^-3}))\\
 =& \sum\limits_{\substack{1\leq m\leq k \\ (m, k)=1}}
 \biggl(
 -\frac{k^6\alpha}{(\alpha-1)^2}+\frac{3k^5\alpha^2}{(\alpha-1)^3}-\frac{3k^5\alpha}{(\alpha-1)^3}-\frac{3k^4\alpha^3}{(\alpha-1)^4}\\
 &+\frac{3k^4\alpha^2}{(\alpha-1)^4}-\frac{3k^4\alpha}{(\alpha-1)^4} -\frac{9k^2\alpha^4}{(\alpha-1)^6}-\frac{9k^2\alpha^2}{(\alpha-1)^6}-\frac{18k^2\alpha^3}{(\alpha-1)^6}
\biggr).
  \end{align*}
 By similar computations we get 
  \begin{align*}
&\sum\limits_{j=1}^{k-1} \sum\limits_{s=1}^{k-1}j^3s^3\sum\limits_{\substack{1\leq m\leq k \\ (m, k)=1}}e^{\frac{2\pi im(j+s)}{k}}\\ =& \sum\limits_{\substack{1\leq m\leq k \\ (m, k)=1}}
\biggl (
 \frac{k^6}{(\alpha-1)^2}-\frac{6k^5\alpha}{(\alpha-1)^3}+\frac{15k^4\alpha^2}{(\alpha-1)^4}+\frac{6k^4\alpha}{(\alpha-1)^4}\\
 &-\frac{18k^3\alpha^3}{(\alpha-1)^5}-\frac{18k^3\alpha^2}{(\alpha-1)^5}+\frac{9k^2\alpha^4}{(\alpha-1)^6}+\frac{9k^2\alpha^2}{(\alpha-1)^6}+\frac{18k^2\alpha^3}{(\alpha-1)^6}
 \biggr).
  \end{align*}
  Hence
  \begin{align}\label{l(3,x)j^3s^3}
 &\frac{4\pi^6}{9k^8}\frac{\phi(k)}{2}
  \left(
  \sum\limits_{j=1}^{k-1} \sum\limits_{s=1}^{k-1}j^3s^3\sum\limits_{\substack{1\leq m\leq k \\ (m, k)=1}}e^{\frac{2\pi im(j-s)}{k}}-\sum\limits_{j=1}^{k-1} \sum\limits_{s=1}^{k-1}j^3s^3\sum\limits_{\substack{1\leq m\leq k \\ (m, k)=1}}e^{\frac{2\pi im(j+s)}{k}}
  \right) \\  
  =& \frac{4\pi^6}{9k^8}\frac{\phi(k)}{2}\sum\limits_{\substack{1\leq m\leq k \\ (m, k)=1}}
  \biggl(
-\frac{k^6}{(\alpha-1)^2}-\frac{k^6\alpha}{(\alpha-1)^2}+\frac{3k^5\alpha^2}{(\alpha-1)^3}+\frac{3k^5\alpha}{(\alpha-1)^3}\nonumber\\&-\frac{3k^4\alpha^3}{(\alpha-1)^4}  -\frac{12k^4\alpha^2}{(\alpha-1)^4}-\frac{9k^4\alpha}{(\alpha-1)^4}
 + \frac{18k^3\alpha^3}{(\alpha-1)^5}+\frac{18k^3\alpha^2}{(\alpha-1)^5}-\frac{18k^2\alpha^4}{(\alpha-1)^6}\nonumber\\&-\frac{36k^2\alpha^3}{(\alpha-1)^6}-\frac{18k^2\alpha^2}{(\alpha-1)^6}
 \biggr).\nonumber
  \end{align}
  Now we evaluate each term by term.
  \begin{align*}
 & \sum\limits_{\substack{1\leq m\leq k \\ (m, k)=1}}-\frac{k^6}{(\alpha-1)^2}\\
  &=-k^6\sum\limits_{\substack{1\leq m\leq k \\ (m, k)=1}}\frac{1}{(e^{\frac{2\pi im}{k}}-1)^2}\times\frac{e^{\frac{-2\pi im}{k}}}{e^{\frac{-2\pi im}{k}}}\\
  &=-k^6\sum\limits_{\substack{1\leq m\leq k \\ (m, k)=1}}\frac{e^{\frac{-2\pi im}{k}}}{(e^{\frac{\pi im}{k}}-e^{\frac{-\pi im}{k}})^2}\\
  &=-k^6\sum\limits_{\substack{1\leq m\leq k \\ (m, k)=1}}\frac{\cos(\frac{2\pi m}{k})-i\sin(\frac{2\pi m}{k})}{(2i)^2 \sin^2(\frac{\pi m}{k})}\\
  &=-k^6
  \left[
\frac{-1}{4}  \sum\limits_{\substack{1\leq m\leq k \\ (m, k)=1}}
\left(
\frac{1-2sin^2(\frac{\pi m}{k})}{\sin^2(\frac{\pi m}{k})}
\right)
+\frac{i}{4}\sum\limits_{\substack{1\leq m\leq k \\ (m, k)=1}}
\left(
\frac{2\sin(\frac{\pi m}{k})cos(\frac{\pi m}{k})}{\sin^2(\frac{\pi m}{k})}
\right)
 \right]\\
  &=-k^6
  \left[
  \frac{-1}{4}  \sum\limits_{\substack{1\leq m\leq k \\ (m, k)=1}}\frac{1}{\sin^2(\frac{\pi m}{k})}+\frac{2}{4}\sum\limits_{\substack{1\leq m\leq k \\ (m, k)=1}} 1+\frac{2i}{4}\sum\limits_{\substack{1\leq m\leq k \\ (m, k)=1}} \cot(\frac{\pi m}{k})
  \right].
  \end{align*}
  Use identity (\ref{eqn:1bysin2}) and the fact that
$ \sum\limits_{\substack{1\leq m\leq k \\ (m, k)=1}} \cot(\frac{\pi m}{k})=0$  as $ \cot(\frac{\pi m}{k})$ and $ \cot(\frac{\pi (k-m)}{k})$  cancel each other to get
   \begin{align}\label{1/(a-1)^2}
 \sum\limits_{\substack{1\leq m\leq k \\ (m, k)=1}}-\frac{k^6}{(\alpha-1)^2} 
 = \frac{k^6 }{12}J_2(k)+\frac{-k^6}{2}\phi(k).
    \end{align}
    A similar evaluation gives
  \begin{align}\label{a/(a-1)^2}
  \sum\limits_{\substack{1\leq m\leq k \\ (m, k)=1}}-\frac{k^6\alpha}{(\alpha-1)^2}= 
  \frac{k^6 }{12}J_2(k).
    \end{align}
    Proceeding similarly, we get
     \begin{align*}
  \sum\limits_{\substack{1\leq m\leq k \\ (m, k)=1}}\frac{3k^5\alpha^2}{(\alpha-1)^3}
 =3k^5
  \left[
\frac{-1}{8i}  \sum\limits_{\substack{1\leq m\leq k \\ (m, k)=1}}
\left(
\frac{\cos(\frac{\pi m}{k})}{\sin^3(\frac{\pi m}{k})}
\right)
+\frac{-i}{8i}\sum\limits_{\substack{1\leq m\leq k \\ (m, k)=1}}
\left(
\frac{1}{\sin^2(\frac{\pi m}{k})}
\right)
 \right].
 \end{align*}
 Once again using identity (\ref{eqn:1bysin2}) and
 \begin{align*} 
 \sum\limits_{\substack{1\leq m\leq k \\ (m, k)=1}} \frac{\cos(\frac{\pi m}{k})}{\sin^3(\frac{\pi m}{k})}&= \sum\limits_{\substack{1\leq m\leq k \\ (m, k)=1}} \cot(\frac{\pi m}{k})\csc^2(\frac{\pi m}{k})\\&=\sum\limits_{\substack{1\leq m\leq k \\ (m, k)=1}}\cot(\frac{\pi m}{k}) +\sum\limits_{\substack{1\leq m\leq k \\ (m, k)=1}} \cot^3(\frac{\pi m}{k})
=0,
\end{align*} we get
\begin{align}\label{a^2/(a-1)^3}
   \sum\limits_{\substack{1\leq m\leq k \\ (m, k)=1}}\frac{3k^5\alpha^2}{(\alpha-1)^3} =
 - \frac{k^5 }{8}J_2(k).
    \end{align}
    Similar computations can be performed to get
    \begin{align}\label{a/(a-1)^3}
     \sum\limits_{\substack{1\leq m\leq k \\ (m, k)=1}}\frac{3k^5\alpha}{(\alpha-1)^3}
  =
  \frac{k^5 }{8}J_2(k).
    \end{align}
    Proceeding similarly, we get
    \begin{align}\label{a^3/(a-1)^4}
    & \sum\limits_{\substack{1\leq m\leq k \\ (m, k)=1}}\frac{-3k^4\alpha^3}{(\alpha-1)^4}\\
  &=-3k^4
  \left[
 \frac{1}{16}  \sum\limits_{\substack{1\leq m\leq k \\ (m, k)=1}}
\frac{1}{\sin^4(\frac{\pi m}{k})}-\frac{2}{16}  \sum\limits_{\substack{1\leq m\leq k \\ (m, k)=1}}
\frac{1}{\sin^2(\frac{\pi m}{k})}+\frac{2i}{16}\sum\limits_{\substack{1\leq m\leq k \\ (m, k)=1}}\frac{\cos(\frac{\pi m}{k})}{\sin^3(\frac{\pi m}{k})}
  \right] \nonumber\\
  &=
 - \frac{k^4 }{240}J_4(k)+\frac{k^4}{12}J_2(k) \text{ using identity (\ref{eqn:1bysin4})}. \nonumber
    \end{align}
    It is not very diffcult to evaluate that
    \begin{align}\label{a^2/(a-1)^4}
  \sum\limits_{\substack{1\leq m\leq k \\ (m, k)=1}}\frac{-12k^4\alpha^2}{(\alpha-1)^4}
  =
 - \frac{k^4 }{60}J_4(k)-\frac{k^4}{6}J_2(k).  
    \end{align}
    Also,
    \begin{align}\label{a/(a-1)^4}
  &  \sum\limits_{\substack{1\leq m\leq k \\ (m, k)=1}}\frac{-9k^4\alpha}{(\alpha-1)^4}\\
 &=-9k^4
  \left[
 \frac{1}{16}  \sum\limits_{\substack{1\leq m\leq k \\ (m, k)=1}}
\frac{1}{\sin^4(\frac{\pi m}{k})}-\frac{2}{16}  \sum\limits_{\substack{1\leq m\leq k \\ (m, k)=1}}
\frac{1}{\sin^2(\frac{\pi m}{k})}-\frac{2i}{16}\sum\limits_{\substack{1\leq m\leq k \\ (m, k)=1}}\frac{\cos(\frac{\pi m}{k})}{\sin^3(\frac{\pi m}{k})}
  \right]\nonumber\\
  &=
 - \frac{k^4 }{80}J_4(k)+\frac{k^4}{4}J_2(k)\nonumber.
    \end{align}
    Now
    \begin{align*}
    \sum\limits_{\substack{1\leq m\leq k \\ (m, k)=1}}\frac{18k^3\alpha^3}{(\alpha-1)^5}
  &=18k^3
  \left[
\frac{1}{32i}  \sum\limits_{\substack{1\leq m\leq k \\ (m, k)=1}}
\frac{\cos(\frac{\pi m}{k})}{\sin^5(\frac{\pi m}{k})}
+\frac{1}{32}\sum\limits_{\substack{1\leq m\leq k \\ (m, k)=1}}
\frac{1}{\sin^4(\frac{\pi m}{k})}
 \right].
   \end{align*}
  \begin{align*} 
 \text{ Since }&\sum\limits_{\substack{1\leq m\leq k \\ (m, k)=1}} 
\frac{\cos(\frac{\pi m}{k})}{\sin^5(\frac{\pi m}{k})}
= \sum\limits_{\substack{1\leq m\leq k \\ (m, k)=1}} \cot(\frac{\pi m}{k})\csc^4(\frac{\pi m}{k})\\&=\sum\limits_{\substack{1\leq m\leq k \\ (m, k)=1}}\cot(\frac{\pi m}{k}) +2\sum\limits_{\substack{1\leq m\leq k \\ (m, k)=1}} \cot^3(\frac{\pi m}{k})+\sum\limits_{\substack{1\leq m\leq k \\ (m, k)=1}} \cot^5(\frac{\pi m}{k})
=0,
\end{align*}
we get
 \begin{align}\label{a^3/(a-1)^5}
    \sum\limits_{\substack{1\leq m\leq k \\ (m, k)=1}}\frac{18k^3\alpha^3}{(\alpha-1)^5}
=
  \frac{k^3 }{80}J_4(k)+\frac{k^3}{8}J_2(k)
    \end{align}
    and
    \begin{align}\label{a^2/(a-1)^5}
    \sum\limits_{\substack{1\leq m\leq k \\ (m, k)=1}}\frac{18k^3\alpha^2}{(\alpha-1)^5}
  =
  -\frac{k^3 }{80}J_4(k)-\frac{k^3}{8}J_2(k).   
    \end{align}
    Similar set of computations using the sum $\sum\limits_{\substack{1\leq m\leq k \\ (m, k)=1}}
\frac{1}{\sin^4(\frac{\pi m}{k})}$ yields
    \begin{align}\label{a^4/(a-1)^6}
     \sum\limits_{\substack{1\leq m\leq k \\ (m, k)=1}}\frac{-18k^2\alpha^4}{(\alpha-1)^6}
  =
     \frac{9k^2}{32}  \sum\limits_{\substack{1\leq m\leq k \\ (m, k)=1}}
\frac{1}{\sin^6(\frac{\pi m}{k})}
 - \frac{k^2}{80}
  J_4(k)-\frac{k^2}{8}J_2(k)
    \end{align}
    and
    \begin{align}\label{a^3/(a-1)^6}
      \sum\limits_{\substack{1\leq m\leq k \\ (m, k)=1}}\frac{-36k^2\alpha^3}{(\alpha-1)^6}
  =
 \frac{9k^2}{16}  \sum\limits_{\substack{1\leq m\leq k \\ (m, k)=1}}
\frac{1}{\sin^6(\frac{\pi m}{k})}.
    \end{align}
    Similarly
    \begin{align}\label{a^2/(a-1)^6}
      \sum\limits_{\substack{1\leq m\leq k \\ (m, k)=1}}\frac{-18k^2\alpha^2}{(\alpha-1)^6}
  =  \frac{9k^2}{32}  \sum\limits_{\substack{1\leq m\leq k \\ (m, k)=1}}
\frac{1}{\sin^6(\frac{\pi m}{k})}
 - \frac{k^2}{80}
  J_4(k)-\frac{k^2}{8}J_2(k).
    \end{align}
  Putting back all the expansions into (\ref{l(3,x)j^3s^3}) we get
    \begin{align*}
  &\frac{4\pi^6}{9k^8}\frac{\phi(k)}{2}
  \left(
  \sum\limits_{j=1}^{k-1} \sum\limits_{s=1}^{k-1}j^3s^3\sum\limits_{\substack{1\leq m\leq k \\ (m, k)=1}}e^{\frac{2\pi im(j-s)}{k}}-\sum\limits_{j=1}^{k-1} \sum\limits_{s=1}^{k-1}j^3s^3\sum\limits_{\substack{1\leq m\leq k \\ (m, k)=1}}e^{\frac{2\pi im(j+s)}{k}}
  \right) \\  
=&\frac{4\pi^6}{9k^8}\frac{\phi(k)}{2}
\biggl[
-\frac{k^6}{2}\phi(k)+\frac{k^6 }{6}J_2(k)- \frac{k^4 }{30}J_4(k)+\frac{k^4}{6}J_2(k)\\&- \frac{k^2}{40}J_4(k)-\frac{k^2}{4}J_2(k)+ \frac{9k^2}{8}  \sum\limits_{\substack{1\leq m\leq k \\ (m, k)=1}}
\frac{1}{\sin^6(\frac{\pi m}{k})}
\biggr]. 
\end{align*}
To simplify the above further, we have to evaluate the sum $\sum\limits_{\substack{1\leq m\leq k \\ (m, k)=1}}
\frac{1}{\sin^6(\frac{\pi m}{k})}$.

We use equation (\ref{1}) with principal character $\chi=\chi_0$ modulo $k\geq3$ and $r=6$ to arrive at the following formula
\begin{align*}
&\frac{k^6}{\pi^6}L(6, \chi_0)\\
=&-\frac{14}{15}\sum\limits_{\substack{1\leq m\leq k \\ (m, k)=1}}
\frac{1}{\sin^2(\frac{\pi m}{k})}-\frac{5}{2}\sum\limits_{\substack{1\leq m\leq k \\ (m, k)=1}}
\frac{\cos(\frac{2\pi m}{k})}{\sin^4(\frac{\pi m}{k})}+2\sum\limits_{\substack{1\leq m\leq k \\ (m, k)=1}}
\frac{\sin(\frac{3\pi m}{k})}{\sin^5(\frac{\pi m}{k})}\\&+\frac{1}{2}\sum\limits_{\substack{1\leq m\leq k \\ (m, k)=1}}
\frac{\cos(\frac{4\pi m}{k})}{\sin^6(\frac{\pi m}{k})}.
\end{align*}
  Now using elementary trigonometric identities on $\sin3\theta$ and $\cos2\theta$, we get
\begin{align*}
\frac{k^6}{\pi^6}L(6, \chi_0)
=\frac{1}{15}\sum\limits_{\substack{1\leq m\leq k \\ (m, k)=1}}
\frac{1}{\sin^2(\frac{\pi m}{k})}-\frac{1}{2}\sum\limits_{\substack{1\leq m\leq k \\ (m, k)=1}}
\frac{1}{\sin^4(\frac{\pi m}{k})}+\frac{1}{2}\sum\limits_{\substack{1\leq m\leq k \\ (m, k)=1}}
\frac{1}{\sin^6(\frac{\pi m}{k})}.
\end{align*}
Using $ \frac{k^6}{\pi^6}L(6, \chi_0)=\frac{k^6}{\pi^6}\zeta(6)\prod\limits_{\substack{p|n\\p\text{ prime}}}\left(1-\frac{1}{p^6}\right)=\frac{J_6(k)}{945}$ \cite[Theorem 11.7]{tom1976introduction},
we conclude that
\begin{align}\label{1bysin6}
\sum\limits_{\substack{1\leq m\leq k \\ (m, k)=1}}
\frac{1}{\sin^6(\frac{\pi m}{k})}=\frac{2}{945}J_6(k)+\frac{1}{45}J_4(k)+\frac{8}{45}J_2(k).
\end{align}
Therefore
\begin{align}\label{j^3s^3L(3,x)}
 &\frac{4\pi^6}{9k^8}\frac{\phi(k)}{2}
  \left(
  \sum\limits_{j=1}^{k-1} \sum\limits_{s=1}^{k-1}j^3s^3\sum\limits_{\substack{1\leq m\leq k \\ (m, k)=1}}e^{\frac{2\pi im(j-s)}{k}}-\sum\limits_{j=1}^{k-1} \sum\limits_{s=1}^{k-1}j^3s^3\sum\limits_{\substack{1\leq m\leq k \\ (m, k)=1}}e^{\frac{2\pi im(j+s)}{k}}
  \right)\nonumber \\
  &=\frac{4\pi^6}{9k^6}\frac{\phi(k)}{2}
\left(
-\frac{k^4}{2}\phi(k)+\frac{k^4 }{6}J_2(k)- \frac{k^2 }{30}J_4(k)+\frac{k^2}{6}J_2(k)-\frac{1}{20}J_2(k)+
\frac{1}{420}J_6(k)  
\right).
    \end{align}
   Now consider the second term on the RHS of (\ref{|l(3,x)|}) and simplify the first sum in bracket to get
   \begin{align*}
  \sum\limits_{j=1}^{k-1} \sum\limits_{s=1}^{k-1}j^3s\sum\limits_{\substack{1\leq m\leq k \\ (m, k)=1}}e^{\frac{2\pi im(j-s)}{k}}= \sum\limits_{\substack{1\leq m\leq k \\ (m, k)=1}}
 \left(
 \sum\limits_{j=1}^{k-1} j^3e^{\frac{2\pi imj}{k}}
\right)
\left(
 \sum\limits_{s=1}^{k-1}se^{\frac{-2\pi ims}{k}}
 \right)
 \end{align*} which on further simplification using identity (\ref{j^3}) gives
  \begin{align*}
&  \sum\limits_{j=1}^{k-1} \sum\limits_{s=1}^{k-1}j^3s\sum\limits_{\substack{1\leq m\leq k \\ (m, k)=1}}e^{\frac{2\pi im(j-s)}{k}}
 \\=& \sum\limits_{\substack{1\leq m\leq k \\ (m, k)=1}}
 \biggl(
-\frac{k^4}{\alpha-1}-\frac{k^4}{(\alpha-1)^2}+\frac{3k^3\alpha}{(\alpha-1)^2}+\frac{3k^3\alpha}{(\alpha-1)^3}-\frac{3k^2\alpha^2}{(\alpha-1)^3}\\&-\frac{3k^2\alpha}{(\alpha-1)^3}-\frac{3k^2\alpha^2}{(\alpha-1)^4}-\frac{3k^2\alpha}{(\alpha-1)^4}
\biggr).
  \end{align*}
 Similar computations give
  \begin{align*}
 & \sum\limits_{j=1}^{k-1} \sum\limits_{s=1}^{k-1}j^3s\sum\limits_{\substack{1\leq m\leq k \\ (m, k)=1}}e^{\frac{2\pi im(j+s)}{k}}\\&= \sum\limits_{\substack{1\leq m\leq k \\ (m, k)=1}}
 \left(
 \frac{k^4}{(\alpha-1)^2}-\frac{3k^3\alpha}{(\alpha-1)^3}+\frac{(3k^2\alpha^2}{(\alpha-1)^4}+\frac{3k^2\alpha}{(\alpha-1)^4}
 \right).
  \end{align*}
  Thus the second term on the RHS of (\ref{|l(3,x)|}) becomes
  \begin{align*}
  &\frac{4\pi^6}{9k^6}\frac{\phi(k)}{2}
  \left(
  \sum\limits_{j=1}^{k-1} \sum\limits_{s=1}^{k-1}j^3s\sum\limits_{\substack{1\leq m\leq k \\ (m, k)=1}}e^{\frac{2\pi im(j-s)}{k}}-\sum\limits_{j=1}^{k-1} \sum\limits_{s=1}^{k-1}j^3s\sum\limits_{\substack{1\leq m\leq k \\ (m, k)=1}}e^{\frac{2\pi im(j+s)}{k}}
  \right) \\
 =&\frac{4\pi^6}{9k^6}\frac{\phi(k)}{2}
  \sum\limits_{\substack{1\leq m\leq k \\ (m, k)=1}}
  \biggl(
  -\frac{k^4}{\alpha-1}-\frac{2k^4}{(\alpha-1)^2}+\frac{3k^3\alpha}{(\alpha-1)^2}+\frac{(6k^3-3k^2)\alpha}{(\alpha-1)^3}\\&-\frac{3k^2\alpha^2}{(\alpha-1)^3}-\frac{6k^2\alpha}{(\alpha-1)^4}
  -\frac{6k^2\alpha^2}{(\alpha-1)^4}
  \biggr).
  \end{align*}

 From identities (\ref{1/(a-1)^2}), (\ref{a/(a-1)^2}), (\ref{a/(a-1)^3}), (\ref{a^2/(a-1)^3}), (\ref{a/(a-1)^4}), (\ref{a^2/(a-1)^4}) and
\begin{align}\label{1/a-1}
\sum\limits_{\substack{1\leq m\leq k \\ (m, k)=1}}\frac{-k^4}{\alpha-1}&=-k^4\sum\limits_{\substack{1\leq m\leq k \\ (m, k)=1}}\frac{1}{\alpha-1}\\&=-k^4\left[
 \frac{1}{2i} \sum\limits_{\substack{1\leq m\leq k \\ (m, k)=1}}\frac{\cos(\frac{\pi m}{k})}{\sin(\frac{\pi m}{k})}-\frac{1}{2} \sum\limits_{\substack{1\leq m\leq k \\ (m, k)=1}}1
  \right]=\frac{k^4}{2}\phi(k)\nonumber
\end{align}

 we get
\begin{align}\label{j^3sL(3,x)}
&\frac{4\pi^6}{9k^6}\frac{\phi(k)}{2}
  \left(
  \sum\limits_{j=1}^{k-1} \sum\limits_{s=1}^{k-1}j^3s\sum\limits_{\substack{1\leq m\leq k \\ (m, k)=1}}e^{\frac{2\pi im(j-s)}{k}}-\sum\limits_{j=1}^{k-1} \sum\limits_{s=1}^{k-1}j^3s\sum\limits_{\substack{1\leq m\leq k \\ (m, k)=1}}e^{\frac{2\pi im(j+s)}{k}}
  \right)\nonumber \\
   &=\frac{4\pi^6}{9k^6}\frac{\phi(k)}{2}
  \left(
-\frac{k^4}{2}\phi(k)+\frac{k^4}{6}J_2(k)-\frac{k^2}{60}J_4(k)+\frac{k^2}{12}J_2(k)
  \right).
\end{align}
Now we consider the third term on the RHS of (\ref{|l(3,x)|}) and
  interchange the variables $j$ and $s$ to get
  \begin{align}\label{js^3L(3,x)}
  &\frac{4\pi^6}{9k^6}\frac{\phi(k)}{2}
  \left(
  \sum\limits_{j=1}^{k-1} \sum\limits_{s=1}^{k-1}j^3s\sum\limits_{\substack{1\leq m\leq k \\ (m, k)=1}}e^{\frac{2\pi i(k-m)(s-j)}{k}}-\sum\limits_{j=1}^{k-1} \sum\limits_{s=1}^{k-1}j^3s\sum\limits_{\substack{1\leq m\leq k \\ (m, k)=1}}e^{\frac{2\pi im(j+s)}{k}}
  \right) \nonumber\\
 &= \frac{4\pi^6}{9k^6}\frac{\phi(k)}{2}
  \left(
  \sum\limits_{j=1}^{k-1} \sum\limits_{s=1}^{k-1}j^3s\sum\limits_{\substack{1\leq m\leq k \\ (m, k)=1}}e^{\frac{2\pi im(j-s)}{k}}-\sum\limits_{j=1}^{k-1} \sum\limits_{s=1}^{k-1}j^3s\sum\limits_{\substack{1\leq m\leq k \\ (m, k)=1}}e^{\frac{2\pi im(j+s)}{k}}
  \right) \nonumber \\
  &=\frac{4\pi^6}{9k^6}\frac{\phi(k)}{2}
  \left(
-\frac{k^4}{2}\phi(k)+\frac{k^4}{6}J_2(k)-\frac{k^2}{60}J_4(k)+\frac{k^2}{12}J_2(k)
  \right).
  \end{align}
  Now consider the last term on the RHS of (\ref{|l(3,x)|}). 
  First get
  \begin{align*}
  &\sum\limits_{j=1}^{k-1} \sum\limits_{s=1}^{k-1}js\sum\limits_{\substack{1\leq m\leq k \\ (m, k)=1}}e^{\frac{2\pi im(j-s)}{k}}\\
  &=\sum\limits_{\substack{1\leq m\leq k \\ (m, k)=1}}
 \left(
 \sum\limits_{j=1}^{k-1} je^{\frac{2\pi imj}{k}}
\right)
\left( 
 \sum\limits_{s=1}^{k-1}se^{\frac{-2\pi ims}{k}}
 \right)\\
  &=\sum\limits_{\substack{1\leq m\leq k \\ (m, k)=1}}
 \left(
 \frac{k}{\alpha-1}
 \right)
 \left(
 -k-\frac{k}{\alpha-1}
 \right)\\
  &=\sum\limits_{\substack{1\leq m\leq k \\ (m, k)=1}}
 \left(
- \frac{k^2}{\alpha-1}-\frac{k^2}{(\alpha-1)^2}
 \right)
 \end{align*}
 and then
 \begin{align*}
 \sum\limits_{j=1}^{k-1} \sum\limits_{s=1}^{k-1}js\sum\limits_{\substack{1\leq m\leq k \\ (m, k)=1}}e^{\frac{2\pi im(j+s)}{k}}
=\sum\limits_{\substack{1\leq m\leq k \\ (m, k)=1}}
 \left(
 \frac{k^2}{(\alpha-1)^2}
 \right).
  \end{align*}
  So
  \begin{align}\label{jsL(3,x)}
  &\frac{4\pi^6}{9k^4}\frac{\phi(k)}{2}
  \left(
  \sum\limits_{j=1}^{k-1} \sum\limits_{s=1}^{k-1}js\sum\limits_{\substack{1\leq m\leq k \\ (m, k)=1}}e^{\frac{2\pi im(j-s)}{k}}-\sum\limits_{j=1}^{k-1} \sum\limits_{s=1}^{k-1}js\sum\limits_{\substack{1\leq m\leq k \\ (m, k)=1}}e^{\frac{2\pi im(j+s)}{k}}
  \right)\nonumber \\
&=  \frac{4\pi^6}{9k^4}\frac{\phi(k)}{2}
\sum\limits_{\substack{1\leq m\leq k \\ (m, k)=1}}
\left(
- \frac{k^2}{\alpha-1}-\frac{2k^2}{(\alpha-1)^2}
\right)\nonumber \\
 &=\frac{4\pi^6}{9k^6}\frac{\phi(k)}{2}
  \left(
 -\frac{k^4}{2}\phi(k)+ \frac{k^4}{6}J_2(k)
  \right) \text{ using identities } (\ref{1/(a-1)^2}), (\ref{1/a-1}).
\end{align}
From equations (\ref{j^3s^3L(3,x)}), (\ref{j^3sL(3,x)}), (\ref{js^3L(3,x)}), (\ref{jsL(3,x)}) we get
\begin{align*}
 &\sum\limits_{\substack{\chi(\text{mod }  k) \\ \chi \text{ odd}}}|L(3, \chi)|^2 \\
 &=\frac{\pi^6}{90k^6}\phi(k)
 \left(
\frac{1}{21}J_6(k) -J_2(k)
 \right)
\end{align*} which is what we required.
So the average value of $ \sum\limits_{\substack{\chi(\text{mod }  k) \\ \chi \text{ odd}}}|L(3, \chi)|^2$ over all odd characters modulo $k$ is $
 \frac{\pi^6}{45k^6}\left(
\frac{1}{21}J_6(k) -J_2(k)\right).
$
\end{proof}
To compute the sum in theorem \ref{L(4,x)}, we need the following identity.
\begin{lemm}\label{phi_4}
\begin{align*}
\phi_4(n)=\frac{n^4}{5}\phi(n)+\frac{n^3}{3}\prod\limits_{p|n}(1-p)-\frac{n}{30}\prod\limits_{p|n}(1-p^3).
\end{align*}
\end{lemm}
\begin{proof}
It is not very difficult to show that
\begin{align*}
\sum\limits_{d|n}\frac{\phi_k(d)}{d^k}=\frac{1^k+\cdots+n^k}{n^k} \text{    \cite[Chapter 2,exercise 15]{tom1976introduction} }.
\end{align*}

If we write $f(n)=\sum\limits_{j=1}^n \frac{j^4}{n^4}$ and $g(n)=\frac{\phi_4(n)}{n^4}$
 then  $\sum\limits_{d|n}g(d)=f(n)$.

By Mobius inversion we get
\begin{align*}
\frac{\phi_4(n)}{n^4}=&\sum\limits_{d|n}\left(\sum\limits_{j=1}^d \frac{j^4}{d^4}\right)\mu\left(\frac{n}{d}\right)\\
=&\sum\limits_{d|n}\frac{d(d+1)(2d+1)(3d^2+3d-1)}{30}\frac{\mu(\frac{n}{d})}{d^4}\\
=&\frac{1}{5}\sum\limits_{d|n}d\mu\left(\frac{n}{d}\right)+\frac{1}{2}\sum\limits_{d|n}\mu\left(\frac{n}{d}\right)+\frac{1}{3}\sum\limits_{d|n}\frac{\mu(\frac{n}{d})}{d}-\frac{1}{30}\sum\limits_{d|n}\frac{\mu(\frac{n}{d})}{d^3}\\
=&\frac{1}{5}\phi(n)+\frac{1}{3}\sum\limits_{d|n}\frac{\mu(d)}{\frac{n}{d}}-\frac{1}{30}\sum\limits_{d|n}\frac{\mu(d)}{(\frac{n}{d})^3}\\&\text{ since }
 \sum\limits_{d|n}d\mu\left(\frac{n}{d}\right)=\phi(n) \text{ and }\sum\limits_{d|n}\mu\left(\frac{n}{d}\right)=0
\\
=&\frac{1}{5}\phi(n)+\frac{1}{3n}\sum\limits_{d|n}d\mu(d)-\frac{1}{30n^3}\sum\limits_{d|n}d^3\mu(d).
\end{align*}
Therefore
\begin{align*} \phi_4(n)=\frac{n^4}{5}\phi(n)+\frac{n^3}{3}\sum\limits_{d|n}d\mu(d)-\frac{n}{30}\sum\limits_{d|n}d^3\mu(d).
\end{align*}
Let $n=p_1^{a_1}p_2^{a_2}\cdots p_r^{a_r}$ be the prime decomposition of $n$.
\begin{align*}
\sum\limits_{d|n}d\mu(d)&=\mu(1)+p_1\mu(p_1)+\cdots +p_r\mu(p_r)+p_1p_2\mu(p_1p_2)\\&+\cdots +p_1p_2\cdots p_r\mu(p_1p_2\cdots p_r)\\
&=1-\sum p_i+\sum p_ip_j-\cdots +(-1)^rp_1\cdots p_r\\
&=\prod\limits_{p|n}(1-p).
\end{align*}
Similarly
\begin{align*}
\sum\limits_{d|n}d^3\mu(d)
&=1-\sum p_i^3+\sum p_i^3p_j^3-\cdots +(-1)^r(p_1\cdots p_r)^3\\
&=\prod\limits_{p|n}(1-p^3).
\end{align*}
Hence, $\phi_4(n)=\frac{n^4}{5}\phi(n)+\frac{n^3}{3}\prod\limits_{p|n}(1-p)-\frac{n}{30}\prod\limits_{p|n}(1-p^3)$.
\end{proof}
Now we proceed to prove theorem \ref{L(4,x)}.
\begin{proof}[Proof of theorem \ref{L(4,x)}]
First note that
\begin{align*}
 \sum\limits_{\substack{\chi(\text{mod }  k) \\ \chi \text{ even}}}|L(4, \chi)|^2=|L(4, \chi_0)|^2+\sum\limits_{\substack{\chi(\text{mod }  k) \\ \chi \text{ even}\\ \chi \neq \chi_0}}|L(4, \chi)|^2.
\end{align*}
We have $L(4, \chi_0)=\zeta(4)\prod\limits_{\substack{p|n\\p\text{ prime}}}\left(1-\frac{1}{p^4}\right)=\frac{\pi^4J_4(k)}{90k^4}$  \cite[Theorem 11.7]{tom1976introduction}.
Therefore 
\begin{align}\label{1*}
\sum\limits_{\substack{\chi(\text{mod }  k) \\ \chi \text{ even}}}|L(4, \chi)|^2=\frac{\pi^8J_4^2(k)}{8100k^8}+\sum\limits_{\substack{\chi(\text{mod }  k) \\ \chi \text{ even}\\ \chi \neq \chi_0}}|L(4, \chi)|^2.
\end{align}
Put $r=4$ in equation (\ref{1}), we have
\begin{align}\label{2*}
L(4, \chi)= \frac{-\pi^4}{3k}	\biggl[&
B_0S(4, \chi)+4B_1S(3, \chi)+6B_2S(2, \chi)+4B_3S(1, \chi)\\&+B_4S(0, \chi)
\biggr]\nonumber.
\end{align}
Put $r=1$ in equation (\ref{3}), we get 
 \begin{align}\label{4*}
 S(1, \chi)= \sum\limits_{j=1}^k \left(\frac{j}{k}\right)G(j, \chi)=\sum\limits_{j=1}^{k-1} \left(\frac{j}{k}\right)G(j, \chi)=0
 \end{align}
 since $r=1$ and $\chi$ are of opposite parity and $G(k, \chi)=0$	 when $\chi$ is non-principal modulo $k\geq3$.
 
 Also put $r=3$ in equation (\ref{3}), we get 
 \begin{align*}
 0=\sum\limits_{q=0}^{2}\binom{3}{q}
 B_q S(3-q,\chi)=
  B_0S(3, \chi)+3B_1S(2, \chi)+3B_2S(1, \chi).
 \end{align*}
 Therefore 
\begin{align}\label{5*}
 S(3, \chi)=\frac{3}{2}S(2, \chi).
\end{align}
Equations (\ref{2*}), (\ref{4*}), (\ref{5*}) and the fact that $S(0, \chi)=0$	, when $\chi \neq \chi_0$ we get
\begin{align*}
L(4, \chi)&= \frac{-\pi^4}{3k}	\left[ S(4, \chi)-2S(2, \chi)
\right]\\
&=\frac{-\pi^4}{3k^3}	\left[\frac{1}{k^2}\sum\limits_{j=1}^{k-1} j^4G(j, \chi)-2\sum\limits_{j=1}^{k-1} j^2G(j, \chi)
\right].
\end{align*}
\begin{align*}
\sum\limits_{\substack{\chi(\text{mod }  k) \\ \chi \text{ even}\\ \chi \neq \chi_0}}|L(4, \chi)|^2=\frac{\pi^8}{9k^6}\sum\limits_{\substack{\chi(\text{mod }  k) \\ \chi \text{ even}\\ \chi \neq \chi_0}}\left|
\frac{1}{k^2}\sum\limits_{j=1}^{k-1} j^4G(j, \chi)-2\sum\limits_{j=1}^{k-1} j^2G(j, \chi)
\right|^2.
\end{align*}
\begin{align}\label{6*}
&\sum\limits_{\substack{\chi(\text{mod }  k) \\ \chi \text{ even}\\ \chi \neq \chi_0}}|L(4, \chi)|^2\\ \nonumber =&\frac{\pi^8}{9k^6}\sum\limits_{\substack{\chi(\text{mod }  k) \\ \chi \text{ even}}}\left|
\frac{1}{k^2}\sum\limits_{j=1}^{k-1} j^4G(j, \chi)-2\sum\limits_{j=1}^{k-1} j^2G(j, \chi)
\right|^2\\ \nonumber &-\frac{\pi^8}{9k^6}\left|
\frac{1}{k^2}\sum\limits_{j=1}^{k-1} j^4G(j, \chi_0)-2\sum\limits_{j=1}^{k-1} j^2G(j, \chi_0)
\right|^2 \nonumber.
\end{align}
Consider the first term in the above diference.
\begin{align*}
&\frac{\pi^8}{9k^6}\sum\limits_{\substack{\chi(\text{mod }  k) \\ \chi \text{ even}}}\left|\frac{1}{k^2}\sum\limits_{j=1}^{k-1} j^4G(j, \chi)-2\sum\limits_{j=1}^{k-1} j^2G(j, \chi)
\right|^2\\
=&\frac{\pi^8}{9k^6}\sum\limits_{\substack{\chi(\text{mod }  k) \\ \chi \text{ even}}}\left[\frac{1}{k^2}\sum\limits_{j=1}^{k-1} j^4G(j, \chi)-2\sum\limits_{j=1}^{k-1} j^2G(j, \chi)
\right] \times \\&\left[\frac{1}{k^2}\sum\limits_{s=1}^{k-1} s^4\overline{G(s, \chi)}-2\sum\limits_{s=1}^{k-1} s^2\overline{G(s, \chi)}					
\right]\\
=&\frac{\pi^8}{9k^6}\sum\limits_{\substack{\chi(\text{mod }  k) \\ \chi \text{ even}}}\biggl[
\frac{1}{k^4}\sum\limits_{j=1}^{k-1}\sum\limits_{s=1}^{k-1}j^4	s^4G(j, \chi)\overline{G(s, \chi)}-\frac{2}{k^2}\sum\limits_{j=1}^{k-1}\sum\limits_{s=1}^{k-1}j^4	s^2G(j, \chi)\overline{G(s, \chi)}\\&-\frac{2}{k^2}\sum\limits_{j=1}^{k-1}\sum\limits_{s=1}^{k-1}j^2	s^4G(j, \chi)\overline{G(s, \chi)}+4\sum\limits_{j=1}^{k-1}\sum\limits_{s=1}^{k-1}j^2	s^2G(j, \chi)\overline{G(s, \chi)}
\biggr].
\end{align*}
For convenience write $S=\left(
 \sum\limits_{m=1}^{k-1}\chi(m)e^{\frac{2\pi imj}{k}}
 \right)
 \left(
 \sum\limits_{n=1}^{k-1}\overline{\chi(n)}e^{\frac{-2\pi ins}{k}}
 \right)$.
 
 Thus we can rewrite the above equation as
\begin{align}\label{|L(4,x)|}
&\frac{\pi^8}{9k^6}\sum\limits_{\substack{\chi(\text{mod }  k) \\ \chi \text{ even}}}\left|\frac{1}{k^2}\sum\limits_{j=1}^{k-1} j^4G(j, \chi)-2\sum\limits_{j=1}^{k-1} j^2G(j, \chi)
\right|^2\\
=&\frac{\pi^8}{9k^{10}}\sum\limits_{\substack{\chi(\text{mod }  k) \\ \chi \text{ even}}}\sum\limits_{j=1}^{k-1}\sum\limits_{s=1}^{k-1}j^4	s^4S-\frac{2\pi^8}{9k^{8}}\sum\limits_{\substack{\chi(\text{mod }  k) \\ \chi \text{ even}}}\sum\limits_{j=1}^{k-1}\sum\limits_{s=1}^{k-1}j^4	s^2S\nonumber \\
 &-\frac{2\pi^8}{9k^8}\sum\limits_{\substack{\chi(\text{mod }  k) \\ \chi \text{ even}}}\sum\limits_{j=1}^{k-1}\sum\limits_{s=1}^{k-1}j^2	s^4S-\frac{4\pi^8}{9k^{6}}\sum\limits_{\substack{\chi(\text{mod }  k) \\ \chi \text{ even}}}\sum\limits_{j=1}^{k-1}\sum\limits_{s=1}^{k-1}j^2	s^2S \nonumber.
 \end{align}
 Firstly we simplify the term
 \begin{align*}
& \frac{\pi^8}{9k^{10}}\sum\limits_{\substack{\chi(\text{mod }  k) \\ \chi \text{ even}}}\sum\limits_{j=1}^{k-1}\sum\limits_{s=1}^{k-1}j^4	s^4S\\
 &=\frac{\pi^8}{9k^{10}}\sum\limits_{j=1}^{k-1} \sum\limits_{s=1}^{k-1}j^4s^4\sum\limits_{m=1}^{k-1}\sum\limits_{n=1}^{k-1}e^{\frac{2\pi i(mj-ns)}{k}}\sum\limits_{\substack{\chi(\text{mod }  k) \\ \chi \text{ even}}}\chi(mn^{-1})\\
 &=\frac{\pi^8}{18k^{10}}\phi(k)
  \left(
  \sum\limits_{j=1}^{k-1} \sum\limits_{s=1}^{k-1}j^4s^4\sum\limits_{\substack{1\leq m\leq k \\ (m, k)=1}}e^{\frac{2\pi im(j-s)}{k}}+\sum\limits_{j=1}^{k-1} \sum\limits_{s=1}^{k-1}j^4s^4\sum\limits_{\substack{1\leq m\leq k \\ (m, k)=1}}e^{\frac{2\pi im(j+s)}{k}}
  \right)
  \end{align*}
where we used the fact that if $u\equiv \pm1$(mod $k$) then $\sum\limits_{\substack{\chi(\text{mod }  k) \\ \chi \text{ even}}}\chi(u)=\frac{\phi(k)}{2}$ for $k \geq 3$ and otherwise $\sum\limits_{\substack{\chi(\text{mod }  k) \\ \chi \text{ even}}}\chi(u)=0$.

By similar computations we arrive at
 \begin{align}\label{L(4,x)1}
 &\frac{\pi^8}{9k^6}\sum\limits_{\substack{\chi(\text{mod }  k) \\ \chi \text{ even}}}\left|\frac{1}{k^2}\sum\limits_{j=1}^{k-1} j^4G(j, \chi)-2\sum\limits_{j=1}^{k-1} j^2G(j, \chi)
\right|^2\\ \nonumber
=&\frac{\pi^8}{18k^{10}}\phi(k)
  \left(
  \sum\limits_{j=1}^{k-1} \sum\limits_{s=1}^{k-1}j^4s^4\sum\limits_{\substack{1\leq m\leq k \\ (m, k)=1}}e^{\frac{2\pi im(j-s)}{k}}+\sum\limits_{j=1}^{k-1} \sum\limits_{s=1}^{k-1}j^4s^4\sum\limits_{\substack{1\leq m\leq k \\ (m, k)=1}}e^{\frac{2\pi im(j+s)}{k}}
  \right)\\ \nonumber
  &-\frac{\pi^8}{9k^{8}}\phi(k)
  \left(
  \sum\limits_{j=1}^{k-1} \sum\limits_{s=1}^{k-1}j^4s^2\sum\limits_{\substack{1\leq m\leq k \\ (m, k)=1}}e^{\frac{2\pi im(j-s)}{k}}+\sum\limits_{j=1}^{k-1} \sum\limits_{s=1}^{k-1}j^4s^2\sum\limits_{\substack{1\leq m\leq k \\ (m, k)=1}}e^{\frac{2\pi im(j+s)}{k}}
  \right)\\ \nonumber
  &-\frac{\pi^8}{9k^8}\phi(k)
  \left(
  \sum\limits_{j=1}^{k-1} \sum\limits_{s=1}^{k-1}j^2s^4\sum\limits_{\substack{1\leq m\leq k \\ (m, k)=1}}e^{\frac{2\pi im(j-s)}{k}}+\sum\limits_{j=1}^{k-1} \sum\limits_{s=1}^{k-1}j^2s^4\sum\limits_{\substack{1\leq m\leq k \\ (m, k)=1}}e^{\frac{2\pi im(j+s)}{k}}
  \right)\\ \nonumber
   &+\frac{2\pi^8}{9k^{6}}\phi(k)
  \left(
  \sum\limits_{j=1}^{k-1} \sum\limits_{s=1}^{k-1}j^2s^2\sum\limits_{\substack{1\leq m\leq k \\ (m, k)=1}}e^{\frac{2\pi im(j-s)}{k}}+\sum\limits_{j=1}^{k-1} \sum\limits_{s=1}^{k-1}j^2s^2\sum\limits_{\substack{1\leq m\leq k \\ (m, k)=1}}e^{\frac{2\pi im(j+s)}{k}}
  \right) \nonumber.
 \end{align}
 
 We simplify each term in this sum. First we consider the term
 \begin{align*} \frac{\pi^8}{18k^{10}}\phi(k)
  \left(
  \sum\limits_{j=1}^{k-1} \sum\limits_{s=1}^{k-1}j^4s^4\sum\limits_{\substack{1\leq m\leq k \\ (m, k)=1}}e^{\frac{2\pi im(j-s)}{k}}+\sum\limits_{j=1}^{k-1} \sum\limits_{s=1}^{k-1}j^4s^4\sum\limits_{\substack{1\leq m\leq k \\ (m, k)=1}}e^{\frac{2\pi im(j+s)}{k}}
  \right).
 \end{align*}

 For simplicity, write $\alpha=e^{\frac{2\pi im}{k}}$.
  
Then
 \begin{align*}
& \sum\limits_{j=1}^{k-1} \sum\limits_{s=1}^{k-1}j^4s^4\sum\limits_{\substack{1\leq m\leq k \\ (m, k)=1}}e^{\frac{2\pi im(j-s)}{k}}\\
 =&\sum\limits_{\substack{1\leq m\leq k \\ (m, k)=1}}\left(\sum\limits_{j=1}^{k-1}j^4e^{\frac{2\pi imj}{k}}	\right)
 \left(
 \sum\limits_{s=1}^{k-1}s^4e^{\frac{-2\pi ims}{k}}
 \right)\\
 =&\sum\limits_{\substack{1\leq m\leq k \\ (m, k)=1}}
  \left(
 \frac{k^4}{\alpha-1}-\frac{4k^3\alpha}{(\alpha-1)^2}+\frac{6k^2\alpha+6k^2\alpha^2}{(\alpha-1)^3}-\frac{(4k\alpha+16k\alpha^2+4k\alpha^3)}{(\alpha-1)^4}
 \right)\\
 &\left(
 \frac{-k^4\alpha}{\alpha-1}-\frac{4k^3\alpha}{(\alpha-1)^2}-\frac{(6k^2\alpha+6k^2\alpha^2)}{(\alpha-1)^3}-\frac{(4k\alpha+16k\alpha^2+4k\alpha^3)}{(\alpha-1)^4}
 \right)
\\ & (\text{ By identities }(\ref{j^4}) \text{ and}(\ref{j^-4}))
  \text{ which expands to} \\
   =&\sum\limits_{\substack{1\leq m\leq k \\ (m, k)=1}}
 \biggl[
  -\frac{k^8\alpha}{(\alpha-1)^2}-\frac{4k^7\alpha}{(\alpha-1)^3}+\frac{4k^7\alpha^2}{(\alpha-1)^3}-\frac{6k^6\alpha}{(\alpha-1)^4}+\frac{4k^6\alpha^2}{(\alpha-1)^4}-\frac{6k^6\alpha^3}{(\alpha-1)^4}\\
 & -\frac{4k^5\alpha}{(\alpha-1)^5}-\frac{12k^5\alpha^2}{(\alpha-1)^5}+\frac{12k^5\alpha^3}{(\alpha-1)^5}+\frac{4k^5\alpha^4}{(\alpha-1)^5}-\frac{4k^4\alpha^2}{(\alpha-1)^6}+\frac{56k^4\alpha^3}{(\alpha-1)^6}\\&-\frac{4k^4\alpha^4}{(\alpha-1)^6}+\frac{16k^2\alpha^2}{(\alpha-1)^8}+\frac{128k^2\alpha^3}{(\alpha-1)^8}+\frac{288k^2\alpha^4}{(\alpha-1)^8}+\frac{128k^2\alpha^5}{(\alpha-1)^8}+\frac{16k^2\alpha^6}{(\alpha-1)^8}
 \biggr].
 \end{align*}
 By similar computations using equation (\ref{j^4}) we get
 \begin{align*}
& \sum\limits_{j=1}^{k-1} \sum\limits_{s=1}^{k-1}j^4s^4\sum\limits_{\substack{1\leq m\leq k \\ (m, k)=1}}e^{\frac{2\pi im(j+s)}{k}}\\
   =&\sum\limits_{\substack{1\leq m\leq k \\ (m, k)=1}}
 \biggl[
  \frac{k^8}{(\alpha-1)^2}-\frac{8k^7\alpha}{(\alpha-1)^3}+\frac{12k^6\alpha}{(\alpha-1)^4}+\frac{28k^6\alpha^2}{(\alpha-1)^4}-\frac{8k^5\alpha}{(\alpha-1)^5}\\&-\frac{80k^5\alpha^2}{(\alpha-1)^5}-\frac{56k^5\alpha^3}{(\alpha-1)^5}
  +\frac{68k^4\alpha^2}{(\alpha-1)^6}+\frac{200k^4\alpha^3}{(\alpha-1)^6}+\frac{68k^4\alpha^4}{(\alpha-1)^6} -\frac{48k^3\alpha^2}{(\alpha-1)^7}\\&-\frac{240k^3\alpha^3}{(\alpha-1)^7}-\frac{240k^3\alpha^4}{(\alpha-1)^7}-\frac{48k^3\alpha^5}{(\alpha-1)^7}
  +\frac{16k^2\alpha^2}{(\alpha-1)^8}+\frac{256k^2\alpha^4}{(\alpha-1)^8}+\frac{16k^2\alpha^6}{(\alpha-1)^8}\\&+\frac{32k^2\alpha^4}{(\alpha-1)^8}+\frac{128k^2\alpha^3}{(\alpha-1)^8}+\frac{128k^2\alpha^5}{(\alpha-1)^8}
\biggr].
 \end{align*}
 Thus we get
 \begin{align*}
 & \frac{\pi^8}{18k^{10}}\phi(k)
  \left(
  \sum\limits_{j=1}^{k-1} \sum\limits_{s=1}^{k-1}j^4s^4\sum\limits_{\substack{1\leq m\leq k \\ (m, k)=1}}e^{\frac{2\pi im(j-s)}{k}}+\sum\limits_{j=1}^{k-1} \sum\limits_{s=1}^{k-1}j^4s^4\sum\limits_{\substack{1\leq m\leq k \\ (m, k)=1}}e^{\frac{2\pi im(j+s)}{k}}
  \right)\\
   =&\frac{\pi^8}{18k^{10}}\phi(k)\sum\limits_{\substack{1\leq m\leq k \\ (m, k)=1}}\biggl[
\frac{k^8}{(\alpha-1)^2}-\frac{k^8\alpha}{(\alpha-1)^2}-\frac{12k^7\alpha}{(\alpha-1)^3}+\frac{4k^7\alpha^2}{(\alpha-1)^3}+\frac{6k^6\alpha}{(\alpha-1)^4}\\&+\frac{32k^6\alpha^2}{(\alpha-1)^4}-\frac{6k^6\alpha^3}{(\alpha-1)^4}- \frac{12k^5\alpha}{(\alpha-1)^5}-\frac{92k^5\alpha^2}{(\alpha-1)^5}-\frac{44k^5\alpha^3}{(\alpha-1)^5} +\frac{4k^5\alpha^4}{(\alpha-1)^5}\\&+\frac{64k^4\alpha^2}{(\alpha-1)^6}+\frac{256k^4\alpha^3}{(\alpha-1)^6}+\frac{64k^4\alpha^4}{(\alpha-1)^6} -\frac{48k^3\alpha^2}{(\alpha-1)^7}-\frac{240k^3\alpha^3}{(\alpha-1)^7}-\frac{240k^3\alpha^4}{(\alpha-1)^7}\\&-\frac{48k^3\alpha^5}{(\alpha-1)^7}
  +\frac{32k^2\alpha^2}{(\alpha-1)^8}+\frac{256k^2\alpha^3}{(\alpha-1)^8}+\frac{576k^2\alpha^4}{(\alpha-1)^8}+\frac{256k^2\alpha^5}{(\alpha-1)^8}+\frac{32k^2\alpha^6}{(\alpha-1)^8}
  \biggr].
 \end{align*}
 Proceeding using the method used for simplifying (\ref{a^3/(a-1)^5}), we have
  \begin{align*}
 \sum\limits_{\substack{1\leq m\leq k \\ (m, k)=1}}\frac{-12k^5\alpha}{(\alpha-1)^5}
   =&-12k^5\sum\limits_{\substack{1\leq m\leq k \\ (m, k)=1}}\frac{\cos(\frac{3\pi m}{k})-i\sin(\frac{3\pi m}{k})}{(2i)^5 \sin^5(\frac{\pi m}{k})}\\
 =&-12k^5
 \biggl[
\frac{4}{32i}  \sum\limits_{\substack{1\leq m\leq k \\ (m, k)=1}}
\frac{\cos^3(\frac{\pi m}{k})}{\sin^5(\frac{\pi m}{k})}-\frac{3}{32i}  \sum\limits_{\substack{1\leq m\leq k \\ (m, k)=1}}
\frac{\cos(\frac{\pi m}{k})}{\sin^5(\frac{\pi m}{k})}  \\
&-\frac{3}{32}\sum\limits_{\substack{1\leq m\leq k \\ (m, k)=1}}
\frac{1}{\sin^4(\frac{\pi m}{k})}+\frac{1}{8}\sum\limits_{\substack{1\leq m\leq k \\ (m, k)=1}}
\frac{1}{\sin^2(\frac{\pi m}{k})}
\biggr]  \\
 =&-12k^5
  \left[  
-\frac{3}{32}\sum\limits_{\substack{1\leq m\leq k \\ (m, k)=1}}
\frac{1}{\sin^4(\frac{\pi m}{k})}+\frac{1}{8}\sum\limits_{\substack{1\leq m\leq k \\ (m, k)=1}}
\frac{1}{\sin^2(\frac{\pi m}{k})}
 \right]\\
& 
  \text{ since } \sum\limits_{\substack{1\leq m\leq k \\ (m, k)=1}}
\frac{\cos^3(\frac{\pi m}{k})}{\sin^5(\frac{\pi m}{k})}=0 \text{ and } \sum\limits_{\substack{1\leq m\leq k \\ (m, k)=1}}
\frac{\cos(\frac{\pi m}{k})}{\sin^5(\frac{\pi m}{k})}=0\\
 =&\frac{k^5}{40}J_4(k)-\frac{k^5}{4}J_2(k) .
 \end{align*}
 Similarly, using the fact that $\sum\limits_{\substack{1\leq m\leq k \\ (m, k)=1}}
\frac{\cos^3(\frac{\pi m}{k})}{\sin^5(\frac{\pi m}{k})}=\sum\limits_{\substack{1\leq m\leq k \\ (m, k)=1}}\cot^3(\frac{\pi m}{k})\\+\sum\limits_{\substack{1\leq m\leq k \\ (m, k)=1}}\cot^5(\frac{\pi m}{k})=0$ and $  \sum\limits_{\substack{1\leq m\leq k \\ (m, k)=1}}
\frac{\cos(\frac{\pi m}{k})}{\sin^5(\frac{\pi m}{k})}=0$ we get
\begin{align*}
 \sum\limits_{\substack{1\leq m\leq k \\ (m, k)=1}}\frac{4k^5\alpha^4}{(\alpha-1)^5}
 =\frac{k^5}{120}J_4(k)-\frac{k^5}{12}J_2(k).
 \end{align*}
 \begin{align*}
  \sum\limits_{\substack{1\leq m\leq k \\ (m, k)=1}}\frac{-48k^3\alpha^2}{(\alpha-1)^7}=&-48k^3\sum\limits_{\substack{1\leq m\leq k \\ (m, k)=1}}\frac{\alpha^2}{(\alpha-1)^7}\\
  =&-48k^3\sum\limits_{\substack{1\leq m\leq k \\ (m, k)=1}}\frac{\cos(\frac{3\pi m}{k})-i\sin(\frac{3\pi m}{k})}{(2i)^7 \sin^7(\frac{\pi m}{k})} \\
   =&-48k^3
  \biggl[
\frac{-4}{128i}  \sum\limits_{\substack{1\leq m\leq k \\ (m, k)=1}}
\frac{\cos^3(\frac{\pi m}{k})}{\sin^7(\frac{\pi m}{k})}+\frac{3}{128i}  \sum\limits_{\substack{1\leq m\leq k \\ (m, k)=1}}
\frac{\cos(\frac{\pi m}{k})}{\sin^7(\frac{\pi m}{k})}  \\
&+\frac{3}{128}\sum\limits_{\substack{1\leq m\leq k \\ (m, k)=1}}
\frac{1}{\sin^6(\frac{\pi m}{k})}-\frac{1}{32}\sum\limits_{\substack{1\leq m\leq k \\ (m, k)=1}}
\frac{1}{\sin^4(\frac{\pi m}{k})}
 \biggr] \\
 =&-48k^3
  \left[ 
\frac{3}{128}\sum\limits_{\substack{1\leq m\leq k \\ (m, k)=1}}
\frac{1}{\sin^6(\frac{\pi m}{k})}-\frac{1}{32}\sum\limits_{\substack{1\leq m\leq k \\ (m, k)=1}}
\frac{1}{\sin^4(\frac{\pi m}{k})}
 \right]
\end{align*} 
  Since \begin{align*}
  \sum\limits_{\substack{1\leq m\leq k \\ (m, k)=1}}
\frac{\cos^3(\frac{\pi m}{k})}{\sin^7(\frac{\pi m}{k})}=\sum\limits_{\substack{1\leq m\leq k \\ (m, k)=1}}\cot^3(\frac{\pi m}{k})+2\sum\limits_{\substack{1\leq m\leq k \\ (m, k)=1}}\cot^5(\frac{\pi m}{k}) +\sum\limits_{\substack{1\leq m\leq k \\ (m, k)=1}}\cot^7(\frac{\pi m}{k})=0
\end{align*} and
 \begin{align*}
&\sum\limits_{\substack{1\leq m\leq k \\ (m, k)=1}} 
\frac{\cos(\frac{\pi m}{k})}{\sin^7(\frac{\pi m}{k})}\\&=\sum\limits_{\substack{1\leq m\leq k \\ (m, k)=1}}\cot(\frac{\pi m}{k})+3\sum\limits_{\substack{1\leq m\leq k \\ (m, k)=1}}\cot^3(\frac{\pi m}{k})+3\sum\limits_{\substack{1\leq m\leq k \\ (m, k)=1}}\cot^5(\frac{\pi m}{k})+\sum\limits_{\substack{1\leq m\leq k \\ (m, k)=1}}\cot^7(\frac{\pi m}{k})=0
\end{align*} 
Hence
\begin{align*}
\sum\limits_{\substack{1\leq m\leq k \\ (m, k)=1}}\frac{-48k^3\alpha^2}{(\alpha-1)^7}=-\frac{k^3}{420}J_6(k)+\frac{k^3}{720}J_4(k)+\frac{2k^3}{15}J_2(k).
\end{align*}
  Using the fact that $\sum\limits_{\substack{1\leq m\leq k \\ (m, k)=1}} 
\frac{\cos(\frac{\pi m}{k})}{\sin^7(\frac{\pi m}{k})}
=0$ we get
  \begin{align*}
   \sum\limits_{\substack{1\leq m\leq k \\ (m, k)=1}}\frac{-240k^3\alpha^3}{(\alpha-1)^7}
   =-\frac{k^3}{252}J_6(k)-\frac{k^3}{24}J_4(k)-\frac{k^3}{3}J_2(k)
  \end{align*}and
  \begin{align*}
   \sum\limits_{\substack{1\leq m\leq k \\ (m, k)=1}}\frac{-240k^3\alpha^4}{(\alpha-1)^7}
=\frac{k^3}{252}J_6(k)+\frac{k^3}{24}J_4(k)+\frac{k^3}{3}J_2(k).
  \end{align*}
  Now using the fact that $\sum\limits_{\substack{1\leq m\leq k \\ (m, k)=1}}
\frac{\cos^3(\frac{\pi m}{k})}{\sin^7(\frac{\pi m}{k})}=0$ and $ \sum\limits_{\substack{1\leq m\leq k \\ (m, k)=1}}
\frac{\cos(\frac{\pi m}{k})}{\sin^7(\frac{\pi m}{k})}=0$ we get
   \begin{align*}
  \sum\limits_{\substack{1\leq m\leq k \\ (m, k)=1}}\frac{-48k^3\alpha^5}{(\alpha-1)^7}
=\frac{k^3}{420}J_6(k)-\frac{k^3}{720}J_4(k)-\frac{2k^3}{15}J_2(k).
  \end{align*}
  Using  $ \sum\limits_{\substack{1\leq m\leq k \\ (m, k)=1}}
\frac{\cos(\frac{\pi m}{k})}{\sin^7(\frac{\pi m}{k})}=0, \sum\limits_{\substack{1\leq m\leq k \\ (m, k)=1}}
\frac{\cos(\frac{\pi m}{k})}{\sin^5(\frac{\pi m}{k})}=0$ and the identities (\ref{eqn:1bysin4}), (\ref{1bysin6}) we get
  \begin{align*}
  \sum\limits_{\substack{1\leq m\leq k \\ (m, k)=1}}\frac{32k^2\alpha^2}{(\alpha-1)^8}
=&32k^2
  \biggl[
\frac{1}{256}\sum\limits_{\substack{1\leq m\leq k \\ (m, k)=1}}
\frac{1}{\sin^8(\frac{\pi m}{k})}-\frac{1}{32}  \sum\limits_{\substack{1\leq m\leq k \\ (m, k)=1}}
\frac{1}{\sin^6(\frac{\pi m}{k})}\\&+\frac{1}{32}  \sum\limits_{\substack{1\leq m\leq k \\ (m, k)=1}}
\frac{1}{\sin^4(\frac{\pi m}{k})}
\biggr]\\
=&\frac{k^2}{8}\sum\limits_{\substack{1\leq m\leq k \\ (m, k)=1}}
\frac{1}{\sin^8(\frac{\pi m}{k})}-\frac{2k^2}{945}J_6(k)+\frac{2}{45}J_2(k),
  \end{align*}
  \begin{align*}
 \sum\limits_{\substack{1\leq m\leq k \\ (m, k)=1}}\frac{256k^2\alpha^3}{(\alpha-1)^8}
   =k^2 \sum\limits_{\substack{1\leq m\leq k \\ (m, k)=1}}\frac{1}{\sin^8(\frac{\pi m}{k})}-\frac{4k^2}{945}J_6(k)-\frac{2k^2}{45}J_4(k)-\frac{16k^2}{45}J_2(k),
  \end{align*}
  \begin{align*}
 \sum\limits_{\substack{1\leq m\leq k \\ (m, k)=1}}\frac{576k^2\alpha^4}{(\alpha-1)^8}=\frac{9k^2}{8}\sum\limits_{\substack{1\leq m\leq k \\ (m, k)=1}}\frac{1}{\sin^8(\frac{\pi m}{k})},
  \end{align*}
  \begin{align*}
  \sum\limits_{\substack{1\leq m\leq k \\ (m, k)=1}}\frac{256k^2\alpha^5}{(\alpha-1)^8}=k^2 \sum\limits_{\substack{1\leq m\leq k \\ (m, k)=1}}\frac{1}{\sin^8(\frac{\pi m}{k})}-\frac{4k^2}{945}J_6(k)-\frac{2k^2}{45}J_4(k)-\frac{16k^2}{45}J_2(k)
  \end{align*}and
  \begin{align*}
 \sum\limits_{\substack{1\leq m\leq k \\ (m, k)=1}}\frac{32k^2\alpha^6}{(\alpha-1)^8}=\frac{k^2}{8}\sum\limits_{\substack{1\leq m\leq k \\ (m, k)=1}}
\frac{1}{\sin^8(\frac{\pi m}{k})}-\frac{2k^2}{945}J_6(k)+\frac{2}{45}J_2(k).
  \end{align*}
  Now as in the case of $\sum\limits_{\substack{1\leq m\leq k \\ (m, k)=1}}
\frac{1}{\sin^6(\frac{\pi m}{k})}$, we get
\begin{align*}
\sum\limits_{\substack{1\leq m\leq k \\ (m, k)=1}}
\frac{1}{\sin^8(\frac{\pi m}{k})}=\frac{1}{4725}J_8(k)+\frac{8}{2835}J_6(k)+\frac{14}{675}J_4(k) +\frac{16}{105}J_2(k).
\end{align*}
On substituting the above identities and (\ref{1/(a-1)^2}), (\ref{a/(a-1)^2}), (\ref{a/(a-1)^3}), (\ref{a^2/(a-1)^3}), (\ref{a/(a-1)^4}), (\ref{a^2/(a-1)^4}), (\ref{a^3/(a-1)^4}), (\ref{a^2/(a-1)^5}), (\ref{a^3/(a-1)^5}), (\ref{a^2/(a-1)^6}), (\ref{a^3/(a-1)^6}), (\ref{a^4/(a-1)^6}) and simplifying, we get
   \begin{align}\label{j^4s^4L(4,x)}
  &\frac{\pi^8}{18k^{10}}\phi(k)
  \left(
  \sum\limits_{j=1}^{k-1} \sum\limits_{s=1}^{k-1}j^4s^4\sum\limits_{\substack{1\leq m\leq k \\ (m, k)=1}}e^{\frac{2\pi im(j-s)}{k}}+\sum\limits_{j=1}^{k-1} \sum\limits_{s=1}^{k-1}j^4s^4\sum\limits_{\substack{1\leq m\leq k \\ (m, k)=1}}e^{\frac{2\pi im(j+s)}{k}}
  \right)\nonumber \\
     =&\frac{\pi^8}{k^8}\phi(k)
 \biggl[
  \frac{k^6}{36}\phi(k)-\frac{k^5}{27}J_2(k)+\frac{k^4}{405}J_4(k)+\frac{2k^4}{81}J_2(k)+\frac{k^3}{270}J_4(k)
  -\frac{2k^2}{2835}J_6(k)\\&-\frac{k^2}{405}J_4(k)-\frac{4k^2}{405}J_2(k)+\frac{J_8(k)}{18900}+\frac{J_4(k)}{4050}+\frac{2}{567}J_2(k)
 \biggr]\nonumber.
  \end{align}
   Now we consider the term
   \begin{align*}
   \frac{\pi^8}{9k^{8}}\phi(k)
  \left(
  \sum\limits_{j=1}^{k-1} \sum\limits_{s=1}^{k-1}j^4s^2\sum\limits_{\substack{1\leq m\leq k \\ (m, k)=1}}e^{\frac{2\pi im(j-s)}{k}}+\sum\limits_{j=1}^{k-1} \sum\limits_{s=1}^{k-1}j^4s^2\sum\limits_{\substack{1\leq m\leq k \\ (m, k)=1}}e^{\frac{2\pi im(j+s)}{k}}
  \right).
  \end{align*}
  The first term in this expands to
  \begin{align*}
  & \sum\limits_{j=1}^{k-1} \sum\limits_{s=1}^{k-1}j^4s^2\sum\limits_{\substack{1\leq m\leq k \\ (m, k)=1}}e^{\frac{2\pi im(j-s)}{k}}\\
 =&\sum\limits_{\substack{1\leq m\leq k \\ (m, k)=1}}\left(\sum\limits_{j=1}^{k-1}j^4e^{\frac{2\pi imj}{k}}	\right)
 \left(
 \sum\limits_{s=1}^{k-1}s^2e^{\frac{-2\pi ims}{k}}
 \right)\\
 =&\sum\limits_{\substack{1\leq m\leq k \\ (m, k)=1}}
  \left(
 \frac{k^4}{\alpha-1}-\frac{4k^3\alpha}{(\alpha-1)^2}+\frac{6k^2\alpha+6k^2\alpha^2}{(\alpha-1)^3}-\frac{(4k\alpha+16k\alpha^2+4k\alpha^3)}{(\alpha-1)^4}
 \right)\\
& \left(
-k^2- \frac{k^2}{\alpha-1}-\frac{2k\alpha}{(\alpha-1)^2}
 \right)(\text{ By identity }(\ref{j^4}))\\
  =&\sum\limits_{\substack{1\leq m\leq k \\ (m, k)=1}}
\biggl[
 - \frac{k^6}{\alpha-1}-\frac{k^6}{(\alpha-1)^2}+\frac{4k^5\alpha}{(\alpha-1)^2}+\frac{2k^5\alpha}{(\alpha-1)^3}-\frac{6k^4\alpha}{(\alpha-1)^3}\\&-\frac{6k^4\alpha^2}{(\alpha-1)^3}-\frac{6k^4\alpha}{(\alpha-1)^4}+\frac{2k^4\alpha^2}{(\alpha-1)^4}+\frac{4k^3\alpha}{(\alpha-1)^4}+\frac{16k^3\alpha^2}{(\alpha-1)^4}+\frac{4k^3\alpha^3}{(\alpha-1)^4}\\&+\frac{4k^3\alpha}{(\alpha-1)^5}+\frac{4k^3\alpha^2}{(\alpha-1)^5}-\frac{8k^3\alpha^3}{(\alpha-1)^5}+\frac{8k^2\alpha^2}{(\alpha-1)^6}+\frac{32k^2\alpha^3}{(\alpha-1)^6}+\frac{8k^2\alpha^4}{(\alpha-1)^6}
 \biggr].
  \end{align*}
  Similarly the second term expands to
  \begin{align*}
  & \sum\limits_{j=1}^{k-1} \sum\limits_{s=1}^{k-1}j^4s^2\sum\limits_{\substack{1\leq m\leq k \\ (m, k)=1}}e^{\frac{2\pi im(j+s)}{k}}\\
 =&\sum\limits_{\substack{1\leq m\leq k \\ (m, k)=1}}
 \biggl[
\frac{k^6}{(\alpha-1)^2}-\frac{6k^5\alpha}{(\alpha-1)^3}-\frac{6k^4\alpha}{(\alpha-1)^4}+\frac{14k^4\alpha^2}{(\alpha-1)^4}-\frac{4k^3\alpha}{(\alpha-1)^5}\\&-\frac{28k^3\alpha^2}{(\alpha-1)^5}-\frac{16k^3\alpha^3}{(\alpha-1)^5}+\frac{8k^2\alpha^2}{(\alpha-1)^6}+\frac{32k^2\alpha^3}{(\alpha-1)^6}+\frac{8k^2\alpha^4}{(\alpha-1)^6}
 \biggr].
  \end{align*}
  Thus we get
  \begin{align*}
  &\frac{\pi^8}{9k^{8}}\phi(k)
  \left(
  \sum\limits_{j=1}^{k-1} \sum\limits_{s=1}^{k-1}j^4s^2\sum\limits_{\substack{1\leq m\leq k \\ (m, k)=1}}e^{\frac{2\pi im(j-s)}{k}}+\sum\limits_{j=1}^{k-1} \sum\limits_{s=1}^{k-1}j^4s^2\sum\limits_{\substack{1\leq m\leq k \\ (m, k)=1}}e^{\frac{2\pi im(j+s)}{k}}
  \right)\\
    =&\frac{\pi^8}{9k^{8}}\phi(k)\sum\limits_{\substack{1\leq m\leq k \\ (m, k)=1}}
 \biggl[
  - \frac{k^6}{\alpha-1}+\frac{4k^5\alpha}{(\alpha-1)^2}-\frac{4k^5\alpha}{(\alpha-1)^3}-\frac{6k^4\alpha}{(\alpha-1)^3}-\frac{6k^4\alpha^2}{(\alpha-1)^3}\\&+\frac{16k^4\alpha^2}{(\alpha-1)^4}+\frac{4k^3\alpha}{(\alpha-1)^4}+\frac{16k^3\alpha^2}{(\alpha-1)^4}+\frac{4k^3\alpha^3}{(\alpha-1)^4}-\frac{24k^3\alpha^2}{(\alpha-1)^5}-\frac{24k^3\alpha^3}{(\alpha-1)^5}\\&+\frac{16k^2\alpha^2}{(\alpha-1)^6}+\frac{64k^2\alpha^3}{(\alpha-1)^6}+\frac{16k^2\alpha^4}{(\alpha-1)^6}
 \biggr].
  \end{align*}
 Using identities (\ref{a/(a-1)^2}), (\ref{a^2/(a-1)^3}), (\ref{a/(a-1)^3}), (\ref{a^3/(a-1)^4}), (\ref{a^2/(a-1)^4}), (\ref{a/(a-1)^4}), (\ref{a^3/(a-1)^5}), (\ref{a^2/(a-1)^5}), (\ref{a^4/(a-1)^6}), (\ref{a^3/(a-1)^6}), (\ref{a^2/(a-1)^6}), (\ref{1/a-1}), we further modify the above to 

\begin{align}\label{j^4s^2L(4,x)}
&\frac{\pi^8}{9k^{8}}\phi(k)
  \left(
  \sum\limits_{j=1}^{k-1} \sum\limits_{s=1}^{k-1}j^4s^2\sum\limits_{\substack{1\leq m\leq k \\ (m, k)=1}}e^{\frac{2\pi im(j-s)}{k}}+\sum\limits_{j=1}^{k-1} \sum\limits_{s=1}^{k-1}j^4s^2\sum\limits_{\substack{1\leq m\leq k \\ (m, k)=1}}e^{\frac{2\pi im(j+s)}{k}}
  \right)\nonumber \\
   =&\frac{\pi^8}{k^{8}}\phi(k)
  \biggl[
  \frac{k^6}{18}\phi(k)-\frac{k^5}{18}J_2(k)+\frac{k^4}{405}J_4(k)+\frac{2k^4}{81}J_2(k)\\&+\frac{k^3}{270}J_4(k)-\frac{k^2}{2835}J_6(k)-\frac{k^2}{810}J_4(k)-\frac{2k^2}{405}J_2(k)
  \biggr]\nonumber.
\end{align}
Next we consider 
\begin{align*}
\frac{\pi^8}{9k^8}\phi(k)
  \left(
  \sum\limits_{j=1}^{k-1} \sum\limits_{s=1}^{k-1}j^2s^4\sum\limits_{\substack{1\leq m\leq k \\ (m, k)=1}}e^{\frac{2\pi im(j-s)}{k}}-\sum\limits_{j=1}^{k-1} \sum\limits_{s=1}^{k-1}j^2s^4\sum\limits_{\substack{1\leq m\leq k \\ (m, k)=1}}e^{\frac{2\pi im(j+s)}{k}}
  \right).
  \end{align*}
  Interchange the variables $j$ and $s$, we get
  \begin{align}\label{j^2s^4L(4,x)}
  &\frac{\pi^8}{9k^8}\phi(k)
  \left(
  \sum\limits_{j=1}^{k-1} \sum\limits_{s=1}^{k-1}j^4s^2\sum\limits_{\substack{1\leq m\leq k \\ (m, k)=1}}e^{\frac{2\pi i(k-m)(s-j)}{k}}-\sum\limits_{j=1}^{k-1} \sum\limits_{s=1}^{k-1}j^4s^2\sum\limits_{\substack{1\leq m\leq k \\ (m, k)=1}}e^{\frac{2\pi im(j+s)}{k}}
  \right) \nonumber \\
 =& \frac{\pi^8}{9k^8}\phi(k)
  \left(
  \sum\limits_{j=1}^{k-1} \sum\limits_{s=1}^{k-1}j^4s^2\sum\limits_{\substack{1\leq m\leq k \\ (m, k)=1}}e^{\frac{2\pi im(j-s)}{k}}-\sum\limits_{j=1}^{k-1} \sum\limits_{s=1}^{k-1}j^4s^2\sum\limits_{\substack{1\leq m\leq k \\ (m, k)=1}}e^{\frac{2\pi im(j+s)}{k}}
  \right)\nonumber \\
 =&\frac{\pi^8}{k^{8}}\phi(k)
  \biggl[
  \frac{k^6}{18}\phi(k)-\frac{k^5}{18}J_2(k)+\frac{k^4}{405}J_4(k)+\frac{2k^4}{81}J_2(k)+\frac{k^3}{270}J_4(k)\\&-\frac{k^2}{2835}J_6(k)-\frac{k^2}{810}J_4(k)-\frac{2k^2}{405}J_2(k)
 \biggr]\nonumber.
  \end{align}
 By a similar sequence of computations, we may see that
\begin{align*}
& \sum\limits_{j=1}^{k-1} \sum\limits_{s=1}^{k-1}j^2s^2\sum\limits_{\substack{1\leq m\leq k \\ (m, k)=1}}e^{\frac{2\pi im(j-s)}{k}}\\&=\sum\limits_{\substack{1\leq m\leq k \\ (m, k)=1}}
\left(
-\frac{k^4}{\alpha-1}- \frac{k^4}{(\alpha-1)^2}+\frac{2k^3\alpha}{(\alpha-1)^2}+\frac{4k^2\alpha^2}{(\alpha-1)^4}
\right)
\end{align*}  
and
\begin{align*}
 \sum\limits_{j=1}^{k-1} \sum\limits_{s=1}^{k-1}j^2s^2\sum\limits_{\substack{1\leq m\leq k \\ (m, k)=1}}e^{\frac{2\pi im(j+s)}{k}}
 =\sum\limits_{\substack{1\leq m\leq k \\ (m, k)=1}}
\left(
\frac{k^4}{(\alpha-1)^2}-\frac{4k^3\alpha}{(\alpha-1)^3}+\frac{4k^2\alpha^2}{(\alpha-1)^4}
\right).
\end{align*} 
Hence we get
\begin{align}\label{j^2s^2L(4,x)}
&\frac{2\pi^8}{9k^{6}}\phi(k)
  \left(
  \sum\limits_{j=1}^{k-1} \sum\limits_{s=1}^{k-1}j^2s^2\sum\limits_{\substack{1\leq m\leq k \\ (m, k)=1}}e^{\frac{2\pi im(j-s)}{k}}+\sum\limits_{j=1}^{k-1} \sum\limits_{s=1}^{k-1}j^2s^2\sum\limits_{\substack{1\leq m\leq k \\ (m, k)=1}}e^{\frac{2\pi im(j+s)}{k}}
  \right) \nonumber \\
  &=\sum\limits_{\substack{1\leq m\leq k \\ (m, k)=1}}
\left(
-\frac{k^4}{\alpha-1}+\frac{2k^3\alpha}{(\alpha-1)^2}-\frac{4k^3\alpha}{(\alpha-1)^3}+\frac{8k^2\alpha^2}{(\alpha-1)^4}
\right)\nonumber\\
  &=\frac{\pi^8}{k^{8}}\phi(k)
  \left(
  \frac{k^6}{9}\phi(k)-\frac{2k^5}{27}J_2(k)+\frac{k^4}{405}J_4(k)+\frac{2k^4}{81}J_2(k)
  \right) \\& \text{ using identities } (\ref{a/(a-1)^2}), (\ref{a/(a-1)^3}), (\ref{a^2/(a-1)^4}), (\ref{1/a-1})\nonumber.
\end{align}
Hence on substituting (\ref{j^4s^4L(4,x)}), (\ref{j^4s^2L(4,x)}), (\ref{j^2s^4L(4,x)}), (\ref{j^2s^2L(4,x)}) into (\ref{L(4,x)1}), we get
\begin{align*}
&\frac{\pi^8}{9k^6}\sum\limits_{\substack{\chi(\text{mod }  k) \\ \chi \text{ even}}}\left|\frac{1}{k^2}\sum\limits_{j=1}^{k-1} j^4G(j, \chi)-2\sum\limits_{j=1}^{k-1} j^2G(j, \chi)
\right|^2\\
&=\frac{\pi^8}{k^{8}}\phi(k)
\left(
 \frac{k^6}{36}\phi(k)+\frac{J_8(k)}{18900}+\frac{J_4(k)}{4050}+\frac{2}{567}J_2(k)
-\frac{k^3}{270}J_4(k)
\right).
\end{align*}
Now we have to find 
$
\frac{\pi^8}{9k^6}\left|
\frac{1}{k^2}\sum\limits_{j=1}^{k-1} j^4G(j, \chi_0)-2\sum\limits_{j=1}^{k-1} j^2G(j, \chi_0)
\right|^2 
$ for evaluating identity (\ref{6*}).

We have 
\begin{align*}
\sum\limits_{j=1}^{k-1} j^2G(j, \chi_0)=\sum\limits_{j=1}^{k-1} j^2R_k(j)&=-\frac{k^2\phi(k)}{2}+\frac{k^3}{\pi^2}L(2, \chi_0)\text{ \cite[identity 3.22]{alkan2011mean}}\\&=-\frac{k^2\phi(k)}{2}+\frac{k}{6}J_2(k).
\end{align*}
Now to find $\sum\limits_{j=1}^{k-1} j^4G(j, \chi_0)=\sum\limits_{j=1}^{k-1} j^4R_k(j)$.

We use the identity $
R_k(j)=\frac{\phi(k)\mu(\frac{k}{(k, j)})}{\phi(\frac{k}{(k,j)})}
$.
\begin{align*}
\sum\limits_{j=1}^{k} j^4R_k(j)&=\sum\limits_{j=1}^{k-1} j^4\frac{\phi(k)\mu(\frac{k}{(k, j)})}{\phi(\frac{k}{(k,j)})}\\
&=\phi(k)\sum\limits_{d|k}\sum\limits_{\substack{1\leq j \leq k\\(k, j)=d}}\frac{j^4\mu(\frac{k}{d})}{\phi(\frac{k}{d})}\\
&=k^4\phi(k)+\phi(k)\sum\limits_{\substack{d|k\\d\neq k}}\sum\limits_{\substack{1\leq j \leq k\\(k, j)=d}}\frac{j^4\mu(\frac{k}{d})}{\phi(\frac{k}{d})}\\
&=k^4\phi(k)+\phi(k)\sum\limits_{\substack{d|k\\d\neq k}}\frac{d^4\mu(\frac{k}{d})}{\phi(\frac{k}{d})}\sum\limits_{\substack{1\leq j \leq k\\(k, j)=d}}\left(\frac{j}{d}\right)^4.
\end{align*}
Therefore 
\begin{align}\label{R_k}
\sum\limits_{j=1}^{k-1} j^4R_k(j)=\phi(k)\sum\limits_{\substack{d|k\\d\neq k}}\frac{d^4\mu(\frac{k}{d})}{\phi(\frac{k}{d})}\sum\limits_{\substack{1\leq j \leq k\\(k, j)=d}}\left(\frac{j}{d}\right)^4.
\end{align}
 
With $n=\frac{k}{d}$ in lemma (\ref{phi_4}) we get

\begin{align} \label{m^4}
\phi_4\left(\frac{k}{d}\right)=\sum\limits_{\substack{1\leq m \leq \frac{k}{d}\\(m, \frac{k}{d})=1}}m^4=\frac{k^4}{5d^4}\phi\left(\frac{k}{d}\right)+\frac{k^3}{3d^3}\prod\limits_{p|(\frac{k}{d})}(1-p)-\frac{k}{30d}\prod\limits_{p|(\frac{k}{d})}(1-p^3).
\end{align}
Denote by $w(n)$, the number of distinct prime divisors of $n$ and $P_s(n)$, product of $s$\textsuperscript{th} powers of all distinct prime divisors of $n$. If $P(n)=P_1(n)$ we have 
\begin{align*}
\phi\left(\frac{k}{d}\right)=\left(\frac{k}{d}\right)\prod\limits_{p|(\frac{k}{d})}\left(1-\frac{1}{p}\right)&=\left( \frac{k}{d}\right)\left(\frac{\prod\limits_{p|(\frac{k}{d})}(p-1)}{\prod\limits_{p|(\frac{k}{d})}p}\right)\\&=\left( \frac{k}{d}\right)(-1)^{w(\frac{k}{d})}\left(\frac{\prod\limits_{p|(\frac{k}{d})}(1-p)}{P(\frac{k}{d})}\right).
\end{align*}

Therefore
\begin{align*}
 \prod\limits_{p|(\frac{k}{d})}(1-p)=\left(\frac{d}{k}\right)(-1)^{w(\frac{k}{d})}P\left(\frac{k}{d}\right)\phi\left(\frac{k}{d}\right).
\end{align*}
We have
\begin{align*}
J_3\left(\frac{k}{d}\right)=\left(\frac{k}{d}\right)^3\prod\limits_{p|(\frac{k}{d})}\left(1-\frac{1}{p^3}\right)&=\left(\frac{k}{d}\right)^3\left(
\frac{\prod\limits_{p|(\frac{k}{d})}(p^3-1)}{\prod\limits_{p|(\frac{k}{d})}p^3}
\right)\\&=\left( \frac{k^3}{d^3}\right)(-1)^{w(\frac{k}{d})}\left(\frac{\prod\limits_{p|(\frac{k}{d})}(1-p^3)}{P_3(\frac{k}{d})}\right).
\end{align*}
 and so
\begin{align*}
 \prod\limits_{p|(\frac{k}{d})}(1-p^3)=\left( \frac{d^3}{k^3}\right)(-1)^{w(\frac{k}{d})}P_3\left(\frac{k}{d}\right)J_3\left(\frac{k}{d}\right).
\end{align*}
Hence from equation (\ref{m^4}) we get
\begin{align*}
\sum\limits_{\substack{1\leq m \leq \frac{k}{d}\\(m, \frac{k}{d})=1}}m^4=&\frac{k^4}{5d^4}\phi\left(\frac{k}{d}\right)+\frac{k^2}{3d^2}(-1)^{w(\frac{k}{d})}P\left(\frac{k}{d}\right)\phi\left(\frac{k}{d}\right)\\&-\frac{d^2}{30k^2}(-1)^{w(\frac{k}{d})}P_3\left(\frac{k}{d}\right)J_3\left(\frac{k}{d}\right).
\end{align*}
So
\begin{align*}
 \sum\limits_{\substack{1\leq j \leq k\\(k, j)=d}}\left(\frac{j}{d}\right)^4=\sum\limits_{\substack{1\leq m \leq \frac{k}{d}\\(m, \frac{k}{d})=1}}m^4=&\frac{k^4}{5d^4}\phi\left(\frac{k}{d}\right)+\frac{k^2}{3d^2}(-1)^{w(\frac{k}{d})}P\left(\frac{k}{d}\right)\phi\left(\frac{k}{d}\right)\\&-\frac{d^2}{30k^2}(-1)^{w(\frac{k}{d})}P_3\left(\frac{k}{d}\right)J_3\left(\frac{k}{d}\right)
 \end{align*}
 and thus equation (\ref{R_k}) becomes,
\begin{align*}
&\sum\limits_{j=1}^{k-1} j^4R_k(j)\\
=&\phi(k)\sum\limits_{\substack{d|k\\d\neq k}}\frac{d^4\mu(\frac{k}{d})}{\phi(\frac{k}{d})}\biggl[
\frac{k^4}{5d^4}\phi\left(\frac{k}{d}\right)+\frac{k^2}{3d^2}(-1)^{w(\frac{k}{d})}P\left(\frac{k}{d}\right)\phi\left(\frac{k}{d}\right)\\&-\frac{d^2}{30k^2}(-1)^{w(\frac{k}{d})}P_3\left(\frac{k}{d}\right)J_3\left(\frac{k}{d}\right)
\biggr]\\
=&\left(\frac{k^4}{5}\right)\phi(k)\sum\limits_{\substack{d|k\\d\neq 1}}\mu(d)+\left(\frac{k^2}{3}\right)\phi(k)\sum\limits_{\substack{d|k\\d\neq k}}d^2(-1)^{w(\frac{k}{d})}P\left(\frac{k}{d}\right)\mu\left(\frac{k}{d}\right)\\
&-\left(\frac{1}{30k^2}\right)\phi(k)\sum\limits_{\substack{d|k\\d\neq k}}d^6(-1)^{w(\frac{k}{d})}\frac{P_3(\frac{k}{d})J_3(\frac{k}{d})\mu(\frac{k}{d})}{\phi(\frac{k}{d})}.
\end{align*}
 Since
$
 \sum\limits_{\substack{d|k\\d\neq 1}}\mu(d)=-1
$
and $
\sum\limits_{d|k}d^2(-1)^{w(\frac{k}{d})}P(\frac{k}{d})\mu(\frac{k}{d})=k^2\prod\limits_{p|k}\left(1+\frac{1}{p}\right) 
$
we get
\begin{align*}
\sum\limits_{j=1}^{k-1} j^4R_k(j)=&-\left(\frac{k^4}{5}\right)\phi(k)+\left(\frac{k^2}{3}\right)\phi(k)\left(-k^2+k^2\prod\limits_{p|k}\left(1+\frac{1}{p}\right)\right)\\&-\left(\frac{1}{30k^2}\right)\phi(k)\sum\limits_{\substack{d|k\\d\neq k}}d^6(-1)^{w(\frac{k}{d})}\frac{P_3(\frac{k}{d})J_3(\frac{k}{d})\mu(\frac{k}{d})}{\phi(\frac{k}{d})}.
\end{align*}
Now we have to evaluate $
\sum\limits_{\substack{d|k\\d\neq k}}d^6(-1)^{w(\frac{k}{d})}\frac{P_3(\frac{k}{d})J_3(\frac{k}{d})\mu(\frac{k}{d})}{\phi(\frac{k}{d})}$.

Consider 
$\sum\limits_{d|k}d^6(-1)^{w(\frac{k}{d})}\frac{P_3(\frac{k}{d})J_3(\frac{k}{d})\mu(\frac{k}{d})}{\phi(\frac{k}{d})}
$.
Let $f(m)=m^6$ and $g(m)=(-1)^{w(m)}\frac{P_3(m)J_3(m)\mu(m)}{\phi(m)}$.

Thus \begin{align*}\sum\limits_{d|k}d^6(-1)^{w(\frac{k}{d})}\frac{P_3(\frac{k}{d})J_3(\frac{k}{d})\mu(\frac{k}{d})}{\phi(\frac{k}{d})}=(f*g)(k)
\end{align*}where $*$ denotes the usual Dirichlet's product.

Since $f$ and $g$ are multiplicative, \begin{align*}(f*g)(k)=(f*g)\left(\prod\limits_{p|k}p^m\right)=\prod\limits_{p|k}(f*g)(p^m)
\end{align*} where $p^m$ is the exact power of a prime $p$ dividing $k$.
\begin{align*}
&\sum\limits_{d|k}d^6(-1)^{w(\frac{k}{d})}\frac{P_3\left(\frac{k}{d}\right)J_3\left(\frac{k}{d}\right)\mu\left(\frac{k}{d}\right)}{\phi(\frac{k}{d})}\\
=& \prod\limits_{p|k}\left(\sum\limits_{d|k}d^6(-1)^{w(\frac{p^m}{d})}\frac{P_3(\frac{p^m}{d})J_3(\frac{p^m}{d})\mu(\frac{p^m}{d})}{\phi(\frac{p^m}{d})}\right)\\
=& \prod\limits_{p|k}\left((p^m-1)^6p^3\left(\frac{1-\frac{1}{p^3}} {p-1}\right)+(p^m)^6
\right)\\& (\text{  since sum is non-zero if } d=p^{m-1} \text{ or } d=p^m )\\
=& \prod\limits_{p|k}\left(
p^{6m-6}\cdot p^3(p^2+p+1)+p^{6m}
\right)\\
=& \left(\prod\limits_{p|k}p^{6m}\right)\left( \prod\limits_{p|k}\left(1+\frac{1}{p}+\frac{1}{p^2}+\frac{1}{p^3}\right)\right)\\
=&k^6\prod\limits_{p|k}\left(1+\frac{1}{p}+\frac{1}{p^2}+\frac{1}{p^3}\right).
\end{align*}
Therefore
\begin{align*}
 \sum\limits_{\substack{d|k\\d\neq k}}d^6(-1)^{w(\frac{k}{d})}\frac{P_3(\frac{k}{d})J_3(\frac{k}{d})\mu(\frac{k}{d})}{\phi(\frac{k}{d})}= -k^6 +k^6\prod\limits_{p|k}\left(1+\frac{1}{p}+\frac{1}{p^2}+\frac{1}{p^3}\right).
\end{align*}
Hence
\begin{align*}
&\sum\limits_{j=1}^{k-1} j^4R_k(j)\\=&-\left(\frac{k^4}{5}\right)\phi(k)+\left(\frac{k^2}{3}\right)\phi(k)\left(-k^2+k^2\prod\limits_{p|k}\left(1+\frac{1}{p}\right)\right)\\&-\left(\frac{1}{30k^2}\right)\phi(k)\left(-k^6 +k^6\prod\limits_{p|k}\left(1+\frac{1}{p}+\frac{1}{p^2}+\frac{1}{p^3}\right)\right)\\
=&-\frac{k^4}{2}\phi(k)+\frac{k^5}{3}\prod\limits_{p|k}\left(1-\frac{1}{p^2}\right)-\frac{k^5}{30}\prod\limits_{p|k}\left(1-\frac{1}{p^4}\right)\\
=&-\frac{k^4}{2}\phi(k)+\frac{k^5}{3}\left(\frac{L(2,\chi_0)}{\zeta(2)}\right)-\frac{k^5}{30}\left(\frac{L(4,\chi_0)}{\zeta(4)}\right) \text{ \cite[Theorem 11.7]{tom1976introduction}}\\
=&-\frac{k^4}{2}\phi(k)+\frac{k^5}{3}\left(\frac{6}{\pi^2}\right)\left(\frac{\pi^2J_2(k)}{6k^2}\right)-\frac{k^5}{30}\left(\frac{90}{\pi^4}\right)\left(\frac{\pi^4J_4(k)}{90k^4}\right)\\
=&-\frac{k^4}{2}\phi(k)+\frac{k^3}{3}J_2(k)-\frac{k}{30}J_4(k).
\end{align*}
Thus,
\begin{align*}
&\frac{\pi^8}{9k^6}\left|
\frac{1}{k^2}\sum\limits_{j=1}^{k-1} j^4G(j, \chi_0)-2\sum\limits_{j=1}^{k-1} j^2G(j, \chi_0)
\right|^2\\&=\frac{\pi^8}{9k^6}\left|
-\frac{k^2}{2}\phi(k)+\frac{k}{3}J_2(k)-\frac{1}{30k}J_4(k)+k^2\phi(k)-\frac{k}{3}J_2(k)
\right|^2\\
&=\frac{\pi^8}{9k^6}\left|
\frac{k^2}{2}\phi(k)-\frac{1}{30k}J_4(k)
\right|^2\\
&=\frac{\pi^8}{k^8}\phi(k)\left(
\frac{k^6}{36}\phi(k)-\frac{k^3}{270}\phi(k)J_4(k)+\frac{1}{8100}\frac{(J_4(k))^2}{\phi(k)}
\right).
\end{align*}
Coming back to equation (\ref{6*}) 
\begin{align*}
&\sum\limits_{\substack{\chi(\text{mod }  k) \\ \chi \text{ even}\\ \chi \neq \chi_0}}|L(4, \chi)|^2\\=&\frac{\pi^8}{k^8}\phi(k)\biggl[
\frac{k^6}{36}\phi(k)+\frac{J_8(k)}{18900}+\frac{J_4(k)}{4050}+\frac{2}{567}J_2(k)
-\frac{k^3}{270}J_4(k)\\&-\frac{k^6}{36}\phi(k)+\frac{k^3}{270}\phi(k)J_4(k)-\frac{1}{8100}\frac{(J_4(k))^2}{\phi(k)}
\biggr]\\
=&\frac{\pi^8}{k^8}\phi(k)\left(\frac{J_8(k)}{18900}+\frac{J_4(k)}{4050}+\frac{2}{567}J_2(k)-\frac{1}{8100}\frac{(J_4(k))^2}{\phi(k)}
\right).
\end{align*}

Hence, equation (\ref{1*}) becomes
\begin{align*}
\sum\limits_{\substack{\chi(\text{mod }  k) \\ \chi \text{ even}}}|L(4, \chi)|^2
=\frac{\pi^8}{27k^8}\phi(k)\left(\frac{J_8(k)}{700}+\frac{J_4(k)}{150}+\frac{2}{21}J_2(k)
\right)
\end{align*}which is what we claimed.
 Hence the average value of $|L(4, \chi)|^2$ over all odd characters modulo $k$ is
$
\frac{2\pi^8}{27k^8}\left(\frac{J_8(k)}{700}+\frac{J_4(k)}{150}+\frac{2}{21}J_2(k)\right)$.

   \end{proof}
   
\section{Further directions}
Though it is possible to extend our method to find higher mean square averages of $L(r,\chi)$, the major difficulty arises in finding the value of $\sum\limits_{\substack{1\leq m\leq k \\ (m, k)=1}}
\frac{1}{\sin^{2r}(\frac{\pi m}{k})}$. It is also possible that our method can be applied to the twisted sum case. But one may have to be ready to undertake longer computations with complexity increasing with respect to the size of $r$.

\section{Acknowledgements}
 The first author thanks the Kerala State Council for Science,Technology and Environment, Thiruvananthapuram, Kerala, India for providing financial support for carrying out research work. The second author thanks the University Grants Commission of India for providing financial support for carrying out research work through their Junior Research Fellowship (JRF) scheme.


\begin{thebibliography}{10}

\bibitem{alkan2011mean}
Emre Alkan.
\newblock On the mean square average of special values of $L$-functions.
\newblock {\em Journal of Number Theory}, 131(8):1470--1485, 2011.

\bibitem{alkan2011values}
Emre Alkan.
\newblock Values of Dirichlet $L$-functions, gauss sums and trigonometric sums.
\newblock {\em The Ramanujan Journal}, 26(3):375--398, 2011.

\bibitem{alkan2013averages}
Emre Alkan.
\newblock Averages of values of $L$-series.
\newblock {\em Proceedings of the American Mathematical Society},
  141(4):1161--1175, 2013.

\bibitem{tom1976introduction}
Tom Apostol.
\newblock {\em Introduction to analytic number theory}.
\newblock Springer, 1976.

\bibitem{lin2019mean}
Xin Lin.
\newblock On the mean square value of the Dirichlet $L$-function at positive
  integers.
\newblock {\em Acta Arithmetica}, 189:367--379, 2019.

\bibitem{louboutin1999mean}
St{\'e}pehane Louboutin et~al.
\newblock On the mean value of $|l (1,\chi)|^2$ for odd primitive Dirichlet
  characters.
\newblock {\em Proceedings of the Japan Academy, Series A, Mathematical
  Sciences}, 75(7):143--145, 1999.

\bibitem{louboutin1993quelques}
St{\'e}phane Louboutin.
\newblock Quelques formules exactes pour des moyennes de fonctions l de
  Dirichlet.
\newblock {\em Canadian Mathematical Bulletin}, 36(2):190--196, 1993.

\bibitem{louboutin2001mean}
St{\'e}phane Louboutin.
\newblock The mean value of $|l(k,\chi)|^2$ at positive rational integers
  $k\geq 1$.
\newblock In {\em Colloquium Mathematicum}, volume~90, pages 69--76. Instytut
  Matematyczny Polskiej Akademii Nauk, 2001.

\bibitem{louboutin2015twisted}
St{\'e}phane~R Louboutin.
\newblock Twisted quadratic moments for Dirichlet $L$-functions.
\newblock {\em Bulletin of the Korean Mathematical Society}, 52(6):2095--2105,
  2015.

\bibitem{montgomery2006multiplicative}
Hugh~L Montgomery and Robert~C Vaughan.
\newblock {\em Multiplicative number theory I: Classical theory}, Volume~97.
\newblock Cambridge University Press, 2006.

\bibitem{namboothiri2017certain}
K~Vishnu Namboothiri.
\newblock Certain weighted averages of generalized Ramanujan sums.
\newblock {\em The Ramanujan Journal}, 44(3):531--547, 2017.

\bibitem{okamoto2015mean}
Takuya Okamoto and Tomokazu Onozuka.
\newblock On various mean values of Dirichlet $L$-functions takuya.
\newblock {\em Acta Arith}, 167:101--115, 2015.

\bibitem{okamoto2017mean}
Takuya Okamoto and Tomokazu Onozuka.
\newblock The mean values of the Dirichlet l-functions (analytic number theory
  and related areas).
\newblock (2014):155--160, 2017.

\bibitem{toth2014averages}
L{\'a}szl{\'o} T{\'o}th.
\newblock Averages of ramanujan sums: Note on two papers by E. Alkan.
\newblock {\em The Ramanujan Journal}, 35(1):149--156, 2014.

\bibitem{walum1982exact}
Herbert Walum et~al.
\newblock An exact formula for an average of $L$-series.
\newblock {\em Illinois Journal of Mathematics}, 26(1):1--3, 1982.

\bibitem{wu2012mean}
Zhaoxia Wu and Wenpeng Zhang.
\newblock On the mean values of $L(1, \chi)$.
\newblock {\em Bulletin of the Korean Mathematical Society}, 49(6):1303--1310,
  2012.

\bibitem{zhang2015hybrid}
Wenpeng Zhang and Di~Han.
\newblock A hybrid mean value of Dedekind sums and Kloosterman sums.
\newblock {\em Journal of Number Theory}, 147:861--870, 2015.

\bibitem{zhang2020certain}
Wenpeng Zhang and Di~Han.
\newblock A certain mean square value involving Dirichlet $L$-functions.
\newblock {\em Mathematics}, 8(6):948, 2020.

\bibitem{zhang2003problem}
Wenpeng Zhang, Xu~Zongben, and Yi~Yuan.
\newblock A problem of DH Lehmer and its mean square value formula.
\newblock {\em Journal of Number theory}, 103(2):197--213, 2003.

\end{thebibliography}

\end{document}